\documentclass[10pt,a4paper]{article}  

\usepackage{graphicx,latexsym,euscript,makeidx,color,bm}
\usepackage{amsmath,amsfonts,amssymb,amsthm,thmtools,mathtools,mathrsfs,enumerate}
\usepackage[colorlinks,linkcolor=blue,citecolor=red]{hyperref}

\usepackage{geometry}
\allowdisplaybreaks[4]
  \def\hA{\widehat{A}}   \def\bu{\bar{u}}
   \def\cB{{\cal B}}  
\def\dbC{\mathbb{C}} \def\sC{\mathscr{C}}

 \def\sG{\mathscr{G}}    
\def\dbH{\mathbb{H}}     
     
     \def\bX{\bar{X}}

\def\dbR{\mathbb{R}}   
\def\dbS{\mathbb{S}}   
   
 \def\sU{\mathscr{U}}  
 \def\sV{\mathscr{V}}  
   
 \def\sX{\mathscr{X}}

\def\nid{\,|\,}       \def\lt{\left}          \def\hb{\hbox}
\def\ms{\medskip}     \def\rt{\right}         \def\ae{\text{a.e.}}
          \def\lan{\langle}       
\def\q{\quad}         \def\ran{\rangle}       \def\tr{{\rm tr}}
\def\qq{\qquad}             \def\diag{\hb{diag\,}}
\def\no{\noindent}          
\def\hp{\hphantom}         \def\scp{\scriptscriptstyle}
\def\nn{\nonumber}         \def\scT{\scp T}
\def\rf{\eqref}       \def\Blan{\Big\lan}     
\def\cd{\cdot}        \def\Bran{\Big\ran}     
\def\deq{\triangleq}  \def\({\Big(}           \def\les{\leqslant}
\def\ti{\tilde}       \def\){\Big)}           \def\ges{\geqslant}
\def\wt{\widetilde}   \def\[{\Big[}           
     \def\]{\Big]}           
\def\ss{\smallskip} 
\def\a{\alpha}       \def\l{\lambda}    
\def\b{\beta}        \def\t{\tau}       \def\F{\varPhi}
\def\d{\delta}       \def\th{\theta}    
\def\e{\varepsilon}      \def\L{\varLambda}
\def\f{\varphi}      \def\p{\phi}       \def\Om{\varOmega}
\def\g{\gamma}       \def\si{\sigma}    \def\Si{\varSigma}
\def\i{\infty}             
       \def\vP{\varPi}    

\def\im{{\rm im}}
\def\bp{\begin{pmatrix}}
\def\ep{\end{pmatrix}}
\def\bY{\bar{Y}}
\def\bv{\bar{v}}
\def\bps{\bar{\psi}_{\scT}}
\def\tX{\wt{X}_{\scT}}
\def\tf{\wt{\f}_{\scT}}
\def\tu{\wt{u}_{\scT}}
\newtheoremstyle{thry}
{}      
{}      
{\sl}   
{}      
{\bf}   
{.}     
{.5em}  
{}      
\theoremstyle{thry}

\newtheorem{theorem}{Theorem}[section]
\newtheorem{proposition}[theorem]{Proposition}
\newtheorem{corollary}[theorem]{Corollary}
\newtheorem{lemma}[theorem]{Lemma}

\theoremstyle{definition}

\theoremstyle{remark}
\newtheorem{remark}[theorem]{Remark}

\def\punct{}
\newtheoremstyle{dotless}{}{}{\rm}{}{\bf}{\punct}{.5em}{}
\theoremstyle{dotless}
\newtheorem{taggedassumptionx}{}
\newenvironment{taggedassumption}[1]
 {\renewcommand\thetaggedassumptionx{#1}\taggedassumptionx}
 {\endtaggedassumptionx}

\makeatletter
   
   \@addtoreset{equation}{section}
   \newcommand{\setword}[2]{%
   \phantomsection
   #1\def\@currentlabel{\unexpanded{#1}}\label{#2}%
   }
\makeatother



\begin{document}

\title{\bf Partial Exponential Turnpike Phenomenon in Linear–Convex Optimal Control}
\author{Jingrui Sun\thanks{Department of Mathematics and SUSTech International Center for Mathematics,
                           Southern University of Science and Technology, Shenzhen, Guangdong,
                           518055, China (Email: sunjr@sustech.edu.cn).
                           This author is supported by NSFC grants 12322118 and 12271242, 
                           and by Shenzhen Science and Technology Program grant JCYJ20250604144337051.}  
~~~
Lvning Yuan\thanks{Corresponding author. Department of Mathematics, Southern University of Science and Technology, 
                   Shenzhen, Guangdong, 518055, China (Email: yuanln2026@126.com).} 
}

\date{}

\maketitle 

\no{\bf Abstract.}
This paper studies the long-time behavior of optimal solutions for a class of linear–convex optimal 
control problems.
We focus on a partial exponential turnpike property, established without imposing controllability or 
stabilizability assumptions, where the turnpike behavior holds only for a subset of initial states.
By means of a refined decomposition of the completely uncontrollable dynamics, we derive necessary 
structural conditions for the turnpike property and explicitly characterize the set of feasible 
initial states.
For each such initial state, we associate a static optimization problem whose unique solution 
determines the corresponding steady state–control pair.
For a class of convex stage cost functions, we prove the partial exponential turnpike property 
and quantify the convergence rate of the averaged finite-horizon optimal cost toward the steady 
optimal value.

\ms
\no{\bf Key words.}
Optimal control, linear-convex, exponential turnpike property, integral turnpike property, static optimization.

\ms
\no{\bf MSC codes.}  34H05, 49N05, 93C05.

\section{Introduction}\label{Sec:Intro}

The turnpike phenomenon is a fundamental long-time behavior in optimal control and optimization problems. 
It asserts that, for sufficiently long time horizons, optimal state trajectories and controls remain close 
to a steady state (or steady state–control pair) for most of the time interval, except for short transient 
phases near the initial and terminal times. Originating in economic growth theory, this phenomenon was first 
studied by Ramsey \cite{Ramsey1928} and von Neumann \cite{Neumann1945}, and later termed the {\it turnpike} 
by Dorfman--Samuelson--Solow \cite{Dorfman-Samuelson-Solow1958}, reflecting the idea that optimal trajectories 
spend most of their time traveling along a common ``highway". 

\ms 

The study of turnpike properties is motivated by both theoretical and practical considerations. 
From a theoretical perspective, turnpike results reveal a deep connection between finite-horizon 
optimal control problems and their associated static or infinite-horizon counterparts, thereby 
providing a structural understanding of long-time optimal behavior. From a practical and 
computational viewpoint, the turnpike phenomenon implies that one only needs to accurately 
resolve the short entry and exit phases of the optimal trajectory, while the interior portion 
can be well approximated by a steady solution. This insight significantly reduces computational 
complexity and improves numerical robustness in long-horizon optimal control problems.

\ms

In recent years, turnpike theory has been extensively developed for a wide range of optimal control problems, 
including deterministic and stochastic systems, continuous- and discrete-time models, as well as finite- and 
infinite-dimensional settings; see, for example, \cite{Breiten-Pfeiffer2020,Lou-Wang2019,Porretta-Zuazua2013,
Sakamoto-Zuazua2021,Sun-Wang-Yong2022,Sun-Yong2024:JDE,Sun-Yong2024:SICON,Trelat-Zhang-Zuazua2018,
Trelat-Zuazua2015,Damm-Grune-Stieler-Worthmann2014,Grune-Guglielmi2018,Trelat-Zhang2018,SCHIEbL2025SCION} 
and the references therein. For more comprehensive accounts, we refer the reader to the monographs 
\cite{Carlson-Haurie-Leizarowitz1991,Zaslavski2019} and the survey articles 
\cite{Zuazua2017,Faulwasser-Grune2022,Geshkovski-Zuazua2022,Trelat-ZuazuaaXirv}.
Many of these works establish {\it global} turnpike properties under additional structural assumptions, 
such as controllability or stabilizability of the system, observability or detectability of the adjoint 
dynamics, or strict dissipativity of the stage cost. 
Under these assumptions---particularly in linear or linear–quadratic settings---the optimal state trajectory 
and control are shown to converge, often at an exponential rate, to a unique steady state–control pair that 
is independent of the initial condition. 

\ms 

However, such structural assumptions are not intrinsic to the solvability of finite-horizon optimal 
control problems. 
Once these assumptions are removed, the classical global turnpike property generally fails. 
In this case, the long-time behavior of optimal solutions may depend sensitively on the initial state: 
different initial conditions may correspond to different steady pairs, and in some situations no steady 
pair exists that attracts all optimal trajectories. 

\ms   

Motivated by this issue, we investigate in this paper the {\it partial exponential turnpike property}
for a class of linear--convex optimal control problems. In contrast to the classical global turnpike
framework, the partial turnpike property is formulated with respect to individual initial states.
Specifically, for a given initial state, we seek a steady state--control pair---possibly depending on
the initial condition---such that the corresponding optimal state trajectory and control remain
exponentially close to this pair over most of the time horizon. Consequently, different initial states
may admit different steady pairs, and only a subset of initial states may exhibit such a property.

\ms

The optimal control problem considered in this paper is formulated as follows. 
Given a finite time horizon $T>0$, consider the controlled linear ordinary differential 
equation (ODE, for short)
\begin{equation}\label{Intro:state}\lt\{\begin{aligned}
\dot X(t) &= AX(t) + Bu(t) +b, \q t\in[0,T], \\
     X(0) &= x,
\end{aligned}\rt.\end{equation}
together with the convex cost functional
\begin{equation}\label{Intro:cost}
J_{\scT}(x;u(\cd)) \deq \int_{0}^{T} f(X(t),u(t))dt,
\end{equation}
where $A\in\dbR^{n\times n}$, $B\in\dbR^{n\times m}$, and $b\in\dbR^n$ are constant coefficients,
and $f(\cdot,\cdot)$ is a convex stage cost function.
Let $\sU[0,T]\deq L^2(0,T;\dbR^m)$ denote the space of square-integrable $\dbR^m$-valued functions on $[0,T]$.  
The associated linear--convex optimal control problem over $[0,T]$ is formulated as follows.

\ms

\noindent\textbf{Problem (LC)$_{\scT}$.}
Given $x\in\dbR^n$, find a control $\bu_{\scT}(\cd)\in\sU[0,T]$ such that 
\begin{equation}\label{Prob:LC}
J_{\scT}(x;\bu_{\scT}(\cd)) = \inf_{u(\cd)\in\sU[0,T]} J_{\scT}(x;u(\cd)) \equiv V_{\scT}(x).
\end{equation}
If such a control exists, it is called an {\it optimal control} for the initial state $x$.
The corresponding trajectory $\bX_{\scT}(\cdot)$ is called the {\it optimal state trajectory},
$(\bX_{\scT}(\cd),\bu_{\scT}(\cd))$ is called the {\it optimal pair}, and $V_{\scT}(\cdot)$ 
is referred to as the {\it value function} of Problem (LC)$_{\scT}$. 

\ms

The main objective of this paper is to characterize all initial states for which 
the exponential turnpike property holds and to establish the corresponding exponential 
estimates for Problem (LC)$_{\scT}$. 
Our analysis does not rely on controllability or stabilizability assumptions. 
Instead, by exploiting a refined decomposition of the state space associated with 
the completely uncontrollable dynamics, we derive sharp structural conditions 
under which turnpike behavior can occur. 
For each feasible initial state, we identify the appropriate static optimization 
problem whose solution yields the associated steady state–control pair.

\ms  

Our main contributions can be summarized as follows.
\begin{enumerate}[(i)]
\item First, in Section \ref{Sec:NC-TP}, we derive necessary structural conditions on the state
equation for the turnpike property to hold at a given initial state $x$ (see \autoref{prop:NC1}), 
and we characterize the set of feasible initial states for which the turnpike property may occur 
(see \autoref{prop:NC2}), based on spectral properties of the completely uncontrollable subsystem. 
No controllability or stabilizability assumptions are imposed on the control system. 
In particular, we show that only a proper subset of initial states may exhibit the turnpike property.
\item Second, in Section \ref{sec:steady-pair}, for each feasible initial state $x$, 
we characterize the associated steady state--control pair. We show that this steady pair 
is the unique solution to a static optimization problem intrinsically related to the feasible 
initial state $x$ (see \autoref{prop:steady-1} and \autoref{prop:ProbO}).
\item Finally, in Section \ref{Sec:PETP}, we show that for a certain class of convex stage 
cost functions, both Problem (LC)$_{\scT}$ and the associated static optimization problem 
admit unique solutions (see \autoref{thm:LC-solvable} and \autoref{thm:O-solvable}). 
Moreover, inspired by the elegant and powerful approach developed in Lou--Wang \cite{Lou-Wang2019}, 
we establish the exponential turnpike property for any feasible initial state $x$ quantify 
the convergence rate of the averaged finite-horizon optimal cost toward the steady optimal value 
(see \autoref{thm:Value-function}, \autoref{thm:PETP-X}, and \autoref{thm:PETP-u}).
\end{enumerate}

\section{Preliminaries}\label{Sec:Pre}

Throughout the paper, all vectors are assumed to be column vectors unless stated otherwise. 
For a function $f:\dbR^n\times\dbR^m\to\dbR$, we denote its first- and second-order partial 
derivatives by $f_{x}$, $f_{u}$, $f_{xx}$, $f_{xu}$, $f_{ux}$, and $f_{uu}$.
The gradient of $f$ is denoted by $\nabla f$, and its Hessian matrix by $\nabla^2 f$.  
The Euclidean space $\dbR^{n\times m}$ of real $n\times m$ matrices is equipped with the 
Frobenius inner product
$$
\lan M,N\ran \deq \tr(M^{\top}N),\q M,N\in\dbR^{n\times m}
$$
and the induced norm $|\cd|$, where $M^\top$ denotes the transpose of $M$ and $\tr(M^{\top}N)$ 
denotes the matrix trace of $M^{\top}N$.
For $A\in\dbR^{n\times n}$, we denote its spectrum by $\si(A)$, its generalized 
eigenspace associated with an eigenvalue $\l$ by $G(\l,A)$, its kernel in $\dbR^n$ by $\ker\!A$,
and its image in $\dbR^n$ by $\im A$. 
Let $\dbS^n$ be the space of real symmetric $n\times n$ matrices, and $\dbS_+^n$ the 
cone of positive definite matrices.
For $M,N\in\dbS^n$, we write $M\ges N$ (resp., $M>N$) if $M-N$ is positive
semidefinite (resp., positive definite).
We denote by $I_n$ the $n\times n$ identity matrix.  
Let $\dbC$ denote the set of complex numbers, and let $\dbC^n$ denote the $n$-dimensional 
complex vector space. 
For $\l\in\dbC$, we denote its real part by $\Re(\l)$, and write  
$$
\dbC^{-} \deq \{\l\in\dbC:\Re(\l)<0\} 
$$
for the open left half-plane. For a Euclidean space $\dbH$ and $p>0$, we define 
\begin{align*}
      C^2(\dbH) &\deq \big\{\f:\dbH \to\dbR \mid \f \text{ is twice continuously differentiable}\big\}, \\
  C([0,T];\dbH) &\deq \big\{\f:[0,T]\to\dbH \mid \f\text{ is continuous}\big\}, \\
L^{p}(0,T;\dbH) &\deq \lt\{\f:[0,T]\to\dbH\Bigm| \int_{0}^{T}|\f(t)|^{p} dt<\i\rt\}.
\end{align*}

\ss

For the reader’s convenience, we present the following simple result for later use.

\begin{lemma}\label{lmm:X-bound}
Let $A,C\in\dbR^{n\times n}$ and $f(\cd)\in L^2(0,T;\dbR^n)$. 
The solution $X(\cd)$ of
$$
\dot X(t) = AX(t) + f(t), \q t\in[0,T]
$$
has the following properties: 
\begin{enumerate}[\rm(i)]
\item For any $0\les t_1\les t_2\les T$, 
      \begin{equation}\label{X-bound:1}
      \sup_{t_1\les t\les t_2}|X(t)|^2 
      \les |X(t_1)|^2 + \big(2|A|+1\big)\int_{t_1}^{t_2}\[|X(s)|^2 + |f(s)|^2\]ds.
      \end{equation}
\item If $(A,C)$ is observable, then for any $\a\in(0,T]$, there exists a constant $K_\a>0$ 
      such that the following observability inequalities hold: 
      \begin{alignat}{3}
      &\label{X-bound:2-1} 
       |X(t)|^2\les K_{\a}\int_{t}^{t+\a}\[|CX(s)|^2+|f(s)|^2\]ds, \q&&\forall t\in[0,T-\a], \\
      &\label{X-bound:2-2}
       |X(t)|^2\les K_{\a}\int_{t-\a}^{t}\[|CX(s)|^2+|f(s)|^2\]ds, \q&&\forall t\in[\a,T].                  
      \end{alignat}
\end{enumerate}
\end{lemma}

\begin{proof}
For (i), differentiating $|X(t)|^2$ gives 
\begin{align*}
\frac{d}{dt}|X(t)|^2 
&= 2\lan X(t),AX(t)+f(t)\ran \les 2|A||X(t)|^2 + |X(t)|^2+|f(t)|^2 \\ 
&\les \big(2|A|+1\big)\big(|X(t)|^2 + |f(t)|^2\big).
\end{align*}
Integrating both sides of the above inequality over $[t_1,t]$ yields 
\begin{align*}
|X(t)|^2 &\les |X(t_1)|^2 + \big(2|A|+1\big)\int_{t_1}^{t}\[|X(s)|^2 + |f(s)|^2\]ds \\
&\les |X(t_1)|^2 + \big(2|A|+1\big)\int_{t_1}^{t_2}\[|X(s)|^2 + |f(s)|^2\]ds, \q\forall t\in[t_1,t_2],
\end{align*}
from which \rf{X-bound:1} follows.  

\ms

For (ii), assume that $(A,C)$ is observable. Then the matrix 
$$
Q(\a) \deq \int_{0}^{\a} e^{rA^{\top}}C^{\top}Ce^{rA}dr 
$$
is invertible. Fix $t\in[0,T-\a]$. By the variation-of-constants formula, 
\begin{equation}\label{Sec2:lmm2-1}
X(t+r)=e^{rA}X(t) + \int_{0}^{r}e^{(r-s)A}f(t+s)ds, \q\forall r\in[0,\a].
\end{equation}
It follows that 
$$
e^{rA^{\top}}C^{\top}C e^{rA}X(t) 
= e^{rA^{\top}}C^{\top}\lt[CX(t+r)-\int_{0}^{r}Ce^{(r-s)A}f(t+s)ds\rt], \q\forall r\in[0,\a].
$$
Integrating both sides of the above over $r\in[0,\a]$ gives 
$$
Q(\a)X(t) = \int_{0}^{\a}e^{rA^{\top}}C^{\top}\lt[CX(t+r)-\int_{0}^{r}Ce^{(r-s)A}f(t+s)ds\rt]dr.
$$
Then by H\"{o}lder's inequality, we obtain
\begin{align}\label{eq:|X(t)|}
|X(t)|^2
&\les |Q(\a)^{-1}|^2\lt|\int_0^{\a}e^{rA^\top}C^\top\lt[CX(t+r)-\int_0^r Ce^{(r-s)A}f(t+s)ds\rt]dr\rt|^2 \nn\\
&\les K_{1,\a}\int_{0}^{\a} \lt[|CX(t+r)|^2+\lt|\int_{0}^{r}Ce^{(r-s)A}f(t+s)ds\rt|^2\rt]dr \nn\\
&\les K_{1,\a}\int_{0}^{\a}\lt[|CX(t+r)|^2+K_{2,\a}\int_{0}^{\a}|f(t+s)|^2ds\rt]dr \nn\\
&\les K_{\a}\int_{t}^{t+\a}\[|C X(s)|^2+|f(s)|^2\]ds. 
\end{align}
This proves \rf{X-bound:2-1}. 
To obtain \rf{X-bound:2-2}, we fix $t\in[\a,T]$ and apply the variation-of-constants formula to get
\begin{equation}\label{Sec2:X(t-r)}
X(t)=e^{rA}X(t-r) + \int_{0}^{r}e^{sA}f(t-s)ds, \q\forall r\in[0,\a].
\end{equation}
Since $(A,C)$ is observable, so is $(-A,C)$. Hence the matrix 
$$
\wt Q(\a) \deq \int_{0}^{\a} e^{-rA^{\top}}C^{\top}Ce^{-rA}dr 
$$
is invertible. Premultiplying \rf{Sec2:X(t-r)} by $e^{-rA^{\top}}C^{\top}Ce^{-rA}$ 
and integrating over $r\in[0,\a]$ yield
$$
\wt Q(\a)X(t) = \int_{0}^{\a}e^{-rA^{\top}}C^{\top}\lt[CX(t-r)+\int_{0}^{r}Ce^{(s-r)A}f(t-s)ds\rt]dr.
$$
Proceeding as in the derivation of \eqref{eq:|X(t)|}, we obtain the desired estimate \rf{X-bound:2-2}.   
\end{proof}

%
%

\section{Structural necessary conditions for turnpike properties}\label{Sec:NC-TP} 

Problem (LC)$_{\scT}$ is said to have the exponential turnpike property 
at an initial state $x\in\dbR^n$ if there exist $(x^*,u^*)\in\dbR^n\times\dbR^m$ 
and constants $K,\l>0$, independent of $T$ (but possibly depending on $x$), such that
\begin{equation}\label{Property:ETP}
|\bX^x_{\scT}(t)-x^*|+|\bu^x_{\scT}(t)-u^*|\les K\[e^{-\l t}+e^{-\l(T-t)}\],\q\ae~t\in[0,T], 
\end{equation}
where $(\bX^x_{\scT}(\cd),\bu^x_{\scT}(\cd))$ is the optimal pair of Problem (LC)$_{\scT}$ 
associated with $x$. This exponential estimate immediately implies the integral turnpike property
\begin{equation}\label{Property:ITP}
\lim_{T\to\i}{1\over T}\int_0^T \[|\bX^x_{\scT}(t)-x^*|^2+|\bu^x_{\scT}(t)-u^*|^2\]dt =0.
\end{equation}
In this section, we derive necessary conditions for the exponential and integral turnpike properties
at a given initial state $x$, and determine all initial states $x$ for which the turnpike property 
can possibly hold.


\ms 

We begin by decomposing the pair $(A,B)$ into its controllable and completely uncontrollable components. 
Let $\sC$ be the controllable subspace of the pair $(A,B)$, that is,  
\begin{equation}\label{Controllable--set}
\sC = \im(B,AB,\cdots,A^{n-1}B),
\end{equation}
and let $k=\dim\sC$. Choose an orthonormal basis $\{v_1,\dots,v_k\}$ of $\sC$ and 
an orthonormal basis $\{v_{k+1},\dots,v_n\}$ of $\sC^{\perp}$. Define 
\begin{equation}\label{Def-P12}
P_1 \deq (v_1,\dots,v_k),\q P_2 \deq (v_{k+1},\dots,v_n), \q P \deq (P_1,P_2),
\end{equation}
so that $P$ is orthogonal. By the Kalman controllable decomposition,  
\begin{equation}\label{Decomposition-AB}
P^{\top}AP = \bp A_{11} & A_{12} \\ 0 & A_{22}\ep, \q  P^{\top}B  = \bp B_1 \\ 0\ep, 
\end{equation}
where $A_{11}\in\dbR^{k\times k}$, $A_{12}\in\dbR^{k\times (n-k)}$, $A_{22}\in\dbR^{(n-k)\times(n-k)}$, $B_1\in\dbR^{k\times m}$, and $(A_{11},B_1)$ is controllable. 

\ms 
 
Now we present the following result, which provides a structural necessary condition for 
the integral turnpike property \rf{Property:ITP},  and hence for the exponential 
turnpike property \rf{Property:ETP}.

\begin{proposition}\label{prop:NC1}
Let $P$ be the orthogonal matrix defined in \rf{Def-P12}.
For an initial state $x\in\dbR^n$, the integral turnpike property \rf{Property:ITP} 
holds only if  
$$
b\in\im(A,B), \q(\text{equivalently, } P_2^\top b\in\im A_{22}).
$$ 
Moreover, any steady pair $(x^*,u^*)$ must satisfy the algebraic equilibrium condition
$$
Ax^*+Bu^*+b=0.
$$
\end{proposition}

\begin{proof}
First, noting that
\begin{align*}
{d\over dt}[\bX^x_{\scT}(t)-x] &= A[\bX^x_{\scT}(t)-x^*] + B[\bu^x_{\scT}(t)-u^*] + (Ax^*+Bu^*+b),
\end{align*}
we obtain by \autoref{lmm:X-bound} (i) that 
\begin{align}\label{25-12-7:eqn1}
|\bX^x_{\scT}(T)-x|^2 
&\les L\int_0^T\[|\bX^x_{\scT}(t)-x^*|^2 + |\bu^x_{\scT}(t)-u^*|^2 + |Ax^*+Bu^*+b|^2\]dt \nn\\
&= L\int_0^T\[|\bX^x_{\scT}(t)-x^*|^2 + |\bu^x_{\scT}(t)-u^*|^2\]dt + LT|Ax^*+Bu^*+b|^2,
\end{align}
for some constant $L>0$ independent of $T$. Moreover,
\begin{align*}
\bX^x_{\scT}(T)-x 
&=\int_{0}^{T}\[A[\bX^x_{\scT}(t)-x^*] + B[\bu^x_{\scT}(t)-u^*]\]dt + (Ax^*+Bu^*+b)T,
\end{align*}
from which it follows that 
\begin{align*}
T^2|Ax^*+Bu^*+b|^2 
&\les 2|\bX^x_{\scT}(T)-x|^2 + 2\lt|\int_0^T\[|\bX^x_{\scT}(t)-x^*|+|\bu^x_{\scT}(t)-u^*|\]dt\rt|^2 \\
&\les 2|\bX^x_{\scT}(T)-x|^2 + 2T \int_0^T\[|\bX^x_{\scT}(t)-x^*|+|\bu^x_{\scT}(t)-u^*|\]^2dt \\
&\les 2|\bX^x_{\scT}(T)-x|^2 + 4T \int_0^T\[|\bX^x_{\scT}(t)-x^*|^2+|\bu^x_{\scT}(t)-u^*|^2\]dt. 
\end{align*}
Dividing the above inequality by $T^2$ and letting $T\to\i$, 
we conclude from \rf{25-12-7:eqn1} and the integral turnpike property \rf{Property:ITP} that  
$$
Ax^* + Bu^* + b = 0,
$$
which, in turn, yields $b\in\im(A,B)$. 
Finally, since $P$ is orthogonal, $b\in\im(A,B)$ is equivalent to 
\begin{equation}\label{25-12-7:eqn2}
\bp P_1^\top b \\ P_2^\top b\ep
=P^\top b \in \im(P^\top AP,P^\top B) 
=\im\bp A_{11} & A_{12} & B_1 \\ 0 & A_{22} & 0\ep.  
\end{equation}
Since $(A_{11},B_1)$ is controllable, $\im(A_{11},A_{12},B_1)=\dbR^k$.
Thus, \rf{25-12-7:eqn2} holds if and only if $ P_2^\top b\in\im A_{22}$.  
\end{proof}

In view of \autoref{prop:NC1}, the condition $b\in \im (A,B)$ is structurally 
necessary for the integral turnpike property. 
Hence, throughout the sequel we impose the following standing assumption.   

\begin{taggedassumption}{(A1)}\label{A1} 
$b\in\im(A,B)$, or equivalently, there exists $c\in\dbR^{n-k}$ such that $P_2^\top b=A_{22}c$. 
\end{taggedassumption}

Next, we provide a necessary structural condition on the initial state $x$ for 
the integral turnpike property to hold. 
To this end, recall that $\dbC^-$ denotes the set of all complex numbers $\l$ 
whose real part is negative, and that for any matrix $M\in\dbR^{n\times n}$, 
we write $\si(M)$ for its spectrum, $\ker\!M$ for its kernel in $\dbR^n$, 
and $G(\l,M)$ for its generalized eigenspace corresponding to an eigenvalue $\l\in\si(M)$. Define   
\begin{equation}\label{Sec3:DefV}
\sG_1\deq \lt(\bigoplus_{\l\in\si(A_{22})\,\cap\,\dbC^-}G(\l,A_{22})\rt)\cap \dbR^{n-k}, \q 
\sG_2\deq \ker\!A_{22} \subseteq \dbR^{n-k}.
\end{equation}

\begin{proposition}\label{prop:NC2}
Let {\rm\ref{A1}} hold. Then for any initial state $x\in\dbR^n$, 
the integral turnpike property \rf{Property:ITP} holds at $x$ only if   
\begin{equation}\label{Sec3:NCx0}
P_2^{\top}x + c \in \sG_1 \oplus \sG_2,  
\end{equation}
where $c\in\dbR^{n-k}$ is any vector satisfying $A_{22}c=P_2^\top b$.  
\end{proposition}

\begin{proof}
Extend the control $\bu^x_{\scT}(\cd)$ to a function $v_{\scT}(\cd)$ on $[0,\i)$ by 
$$
v_{\scT}(t) \deq \bu^x_{\scT}(t){\bf1}_{[0,T]}(t), \q t\in[0,\i),
$$
and let $X(\cd)$ be the solution of 
$$\lt\{\begin{aligned}
\dot X(t) &= AX(t) + Bv_{\scT}(t) +b, \q t\in[0,\i), \\
     X(0) &= x.
\end{aligned}\rt.$$
Then we have
$$
X(t) = \bX^x_{\scT}(t), \q\forall t\in[0,T].  
$$
For $i=1,2$, set 
\begin{equation}\label{Def:Y-d-y}
Y_i(t)\deq P_i^{\top}X(t), \q t\in[0,\i). 
\end{equation}
Using \rf{Decomposition-AB}, we obtain  
\begin{equation}\label{Split-state}\lt\{\begin{aligned}
\dot Y_1(t) &= A_{11}Y_1(t) + B_1v_{\scT}(t) +A_{12}Y_2(t) + P_1^{\top}b, \q  t\in[0,\i), \\
\dot Y_2(t) &= A_{22}Y_2(t) + P_2^{\top}b, \q  t\in[0,\i), \\
     Y_1(0) &=P_1^{\top}x, \q Y_2(0)=P_2^{\top}x.  
\end{aligned}\rt.\end{equation} 
In particular, we observe that $Y_2(\cd)$ does not depend on $T$. 
Since $P_2^\top b=A_{22}c$, we have 
$$
\frac{d}{dt}[Y_2(t)+c]=A_{22}[Y_2(t)+c]. 
$$
Consequently,  
\begin{equation}\label{26-1-17:Y2}
Y_2(t)+c=e^{tA_{22}}(P_2^{\top}x+c), \q t\in[0,\i). 
\end{equation}
Moreover,  
\begin{align}\label{25-12-10:eqn1}
\frac{1}{T}\int_0^T\big|Y_2(t)-P_2^{\top}x^*\big|^2dt
&=\frac{1}{T}\int_0^T\big|P_2^{\top}[\bX^x_{\scT}(t)-x^*]\big|^2dt \nn\\
&\les\frac{1}{T}\int_0^T\big|\bX^x_{\scT}(t)-x^*\big|^2dt \to 0, \q\text{as } T\to\i.  
\end{align}
By \autoref{prop:NC1}, the steady pair $(x^*,u^*)$ satisfies
$$
Ax^*+Bu^*+b=0.
$$ 
Premultiplying by $P^\top$ and using the block decomposition \rf{Decomposition-AB}, we obtain
$$
0=P^\top APP^\top x^* + P^\top Bu^* + P^\top b 
 =\bp A_{11} & A_{12} \\ 0 & A_{22}\ep\bp P_1^\top x^* \\ P_2^\top x^*\ep
 +\bp B_1u^* \\ 0\ep + \bp P_1^\top b \\ P_2^\top b\ep.
$$
Taking the second block row yields
$P_2^\top b=-A_{22}P_2^\top x^*$, and hence
$$
\frac{d}{dt}[Y_2(t)-P_2^\top x^*]= A_{22}Y_2(t) + P_2^{\top}b =A_{22}[Y_2(t)-P_2^\top x^*], \q t\ges0.
$$
By \autoref{lmm:X-bound} (i) and \rf{25-12-10:eqn1}, we see that as $T\to\i$,
$$
\frac{1}{T}|Y_2(T)-P_2^\top x^*|^2 
\les\frac{|P_2^\top x-P_2^\top x^*|^2}{T} + \frac{2|A_{22}|+1}{T}\int_0^T|Y_2(t)-P_2^\top x^*|^2ds\to 0. 
$$
This implies that there exists a constant $K_1>0$ such that
\begin{equation}\label{Y2+c:bound}
|Y_2(t)+c|\les K_1(1+\sqrt{t}), \q\forall t\ges 0. 
\end{equation}
Let $\l_1$, $\l_2$, $\dots$, $\l_\ell$ be the distinct eigenvalues of $A_{22}$ with
algebraic multiplicities $r_1$, $r_2$, $\dots$, $r_\ell$, respectively. 
Then 
$$
\dbC^{n-k}=G(\l_1,A_{22})\oplus \cdots \oplus G(\l_\ell,A_{22}).
$$
Let $Q_j$ be the projection from $\dbC^{n-k}$ onto $G_j(\l_j,A_{22})$ associated with the above direct-sum decomposition.
Then there exists a constant $K_2>0$, such that 
\begin{equation}\label{26-1-17:Qv}
|Q_jv| \les K_2|v|,\q \forall v\in\dbC^{n-k},\q j=1,2,\cdots,\ell.
\end{equation}
Now, for each $j$, define
$
w_j\deq Q_j(P_2^{\top}x+c).
$
Then we have from \eqref{26-1-17:Y2} that 
\begin{equation}\label{26-1-17:QYj}
Q_j(Y_2(t)+c) = e^{tA_{22}}w_j= e^{\l_j t}Z_j(t),
\end{equation}
where 
$$
Z_j(t)\deq\lt[I+t(A_{22}-\l_j I)+ \cdots + \frac{t^{r_j-1}}{(r_j-1)!}(A_{22}-\l_j I)^{r_j-1}\rt]w_j
$$
is a vector-valued polynomial of degree at most $r_j-1$.
From \rf{Y2+c:bound}--\rf{26-1-17:QYj}, we know that the term $e^{\l_j t}Z_j(t)$ grows at most on the order of $\sqrt{t}$. 
Hence, for each $j$ with $w_j\neq0$, this is possible only if
\begin{equation}\label{26-1-17:lambda}
\Re(\l_j)<0,\q \text{or} \q \Re(\l_j)=0 \text{ and } A_{22}w_j=\l_j w_j.
\end{equation}
Consider now the second case in above. 
If $\l_j$ is a nonzero pure imaginary number and $w_j\neq 0$, then 
$$
h(t)\deq Q_j(Y_2(t)-P_2^{\top}x^*)=e^{\l_i t}w_j -Q_j(P_2^{\top}x^*+c)
$$
is a nonzero continuous periodic function with period $\t\deq 2\pi/|\l_j|$, which follows a contradiction by \rf{25-12-10:eqn1} and \rf{26-1-17:Qv} that
$$
0<\frac{1}{\t}\int_{0}^{\t}|h(t)|^2dt=\frac{1}{N\t}\int_{0}^{N\t}|h(t)|^2dt 
\les \frac{K_2}{N\t}\int_{0}^{N\t}\big|Y_2(t)-P_2^{\top}x^*\big|^2dt\to 0, \q \text{as }N\to\i. 
$$
Combing this with \rf{26-1-17:lambda} gives $P_2^\top x+c\in\sG_1\oplus\sG_2$.  
\end{proof}

In what follows, we call an initial state $x\in\dbR^n$ satisfying \rf{Sec3:NCx0} 
a feasible initial state.

\section{The steady pair}\label{sec:steady-pair} 

In this section, we characterize the steady pair $(x^*,u^*)$ associated with a given initial 
state $x$ for which the integral turnpike property \rf{Property:ITP} holds.

\ms 

Recall the notation introduced in \rf{Controllable--set}--\rf{Decomposition-AB} and the spaces 
$$
\sG_1\deq \lt(\bigoplus_{\l\in\si(A_{22})\,\cap\,\dbC^-}G(\l,A_{22})\rt)\cap \dbR^{n-k}, \q 
\sG_2\deq \ker\!A_{22} \subseteq \dbR^{n-k}
$$
as defined in \rf{Sec3:DefV}. For $i=1,2$, let 
\begin{equation}\label{def:Qi}
Q_i:\sG_1\oplus\sG_2\to\sG_i
\end{equation}
denote the projection from $\sG_1\oplus\sG_2$ onto $\sG_i$.  
Let $c\in\dbR^{n-k}$ be a vector such that $A_{22}c=P_2^\top b$. 
Define the set of feasible initial states:
\begin{equation}\label{def:feasible-x}
\sX \deq \bigl\{x\in\dbR^n: P_2^\top x+c\in \sG_1\oplus\sG_2\bigr\}.
\end{equation}
Recall from \autoref{prop:NC2} that the integral turnpike property \rf{Property:ITP} holds 
at $x$ only if $x\in\sX$.

\ms 

First, we present the following result, which shows that the steady state $x^*$ generally 
depends on the initial state $x$.   

\begin{proposition}\label{prop:steady-1}
Let {\rm\ref{A1}} hold, and let  $c\in\dbR^{n-k}$ be any vector satisfying $A_{22}c=P_2^\top b$. 
If the integral turnpike property \rf{Property:ITP} holds at the initial state $x\in\dbR^n$, 
then the corresponding steady state $x^*$ satisfies 
$$
P_2^\top x^*+c = Q_2(P_2^\top x+c). 
$$
\end{proposition}

\begin{proof}
By \autoref{prop:NC2}, we have the unique decomposition  
\begin{equation}\label{Sec3:def-y1y2}
P_2^\top x + c = y_1 + y_2, 
\end{equation}
where $y_i=Q_i(P_2^\top x+c)\in\sG_i$, $i=1,2$.
Consider the function $Y_2(\cd)$ in \rf{Split-state}.  
Form the proof of \autoref{prop:NC2}, we see that 
$$
Y_2(t)+c = e^{tA_{22}}(P_2^{\top}x+c) =  e^{tA_{22}}y_1 +  e^{tA_{22}}y_2, \q t\in[0,\i). 
$$
Since $y_1\in\sG_1$ and $y_2\in\sG_2=\ker\!A_{22}$, it follows that 
$$
\lim_{t\to\i}[Y_2(t)+c] =  \lim_{t\to\i}e^{tA_{22}}y_1 + y_2 = y_2, 
$$
and hence 
$$
\lim_{T\to\i}{1\over T}\int_0^T |Y_2(t)+c-y_2|^2 dt =0. 
$$
Consequently, by using \rf{25-12-10:eqn1} we obtain
\begin{align*}
|P_2^\top x^*+c-y_2|^2 
&= \lim_{T\to\i}{1\over T}\int_0^T |P_2^\top x^*-Y_2(t) + Y_2(t)+c-y_2|^2 dt \\
&\les 2\lim_{T\to\i}\lt[{1\over T}\int_0^T|P_2^\top x^*-Y_2(t)|^2dt+{1\over T}\int_0^T|Y_2(t)+c-y_2|^2 dt\rt] \\
&= 0.
\end{align*}
This completes the proof.  
\end{proof}

Next, for any feasible initial state $x\in\sX$, we define the $x$-dependent set
$$
\sV_{x}\deq \bigl\{(z,u)\in\dbR^n\times\dbR^m \nid Az+ Bu + b=0,~ P_2^\top z+c=Q_2(P_2^\top x+c)\bigr\} 
$$
and introduce the following static optimization problem.  

\ms

{\bf Problem (O).} For a given $x\in\sX$, find $(x^*,u^*)\in\sV_{x}$ such that 
\begin{equation}\label{Prob:Ox0}
f(x^*,u^*)=\inf_{(z,u)\in\sV_x} f(z,u)\equiv V^*(x).
\end{equation}

We present the following result, which shows that if the integral turnpike property 
\rf{Property:ITP} holds at a feasible initial state $x\in\sX$, then the associated steady pair 
$(x^*,u^*)$ solves Problem (O) corresponding to this initial state $x$.

\begin{proposition}\label{prop:ProbO}
Let {\rm\ref{A1}} hold and let $f:\dbR^n\times\dbR^m\to\dbR$ be convex.  
If the integral turnpike property \rf{Property:ITP} holds at an initial state $x\in\sX$, 
then the associated steady pair $(x^*,u^*)\in\sV_x$, and 
\begin{equation}\label{eq:f*<f}
f(x^*,u^*) \les f(z,u), \q\forall (z,u)\in\sV_x. 
\end{equation}
\end{proposition}

\begin{proof}
From \autoref{prop:NC1} and \autoref{prop:steady-1}, we see that $(x^*,u^*)\in\sV_x$.  
The proof of \rf{eq:f*<f} is divided into the following steps.

\ms 

{\it Step 1.} Let $(\bX_{\scT}^x(t),\bu_{\scT}^x(t))$ be the optimal pair of 
Problem (LC)$_{\scT}$ associated with $x$. We conclude that  
\begin{equation}\label{lower-bound}
f(x^*,u^*) \les \liminf_{T\to\i}\frac{1}{T}\int_0^Tf(\bX_{\scT}^x(t),\bu_{\scT}^x(t))dt.
\end{equation}

Indeed, by H\"{o}lder's inequality, 
\begin{align*}
&\lt|\frac{1}{T}\int_0^T\bX_{\scT}^x(t)dt-x^*\rt| + \lt|\frac{1}{T}\int_0^T\bu_{\scT}^x(t)dt-u^*\rt| \\
&\q\les \frac{1}{T}\int_0^T\[|\bX_{\scT}^x(t)-x^*| + |\bu_{\scT}^x(t)-u^*|\]dt \\
&\q\les \frac{1}{T}\lt\{T\int_0^T\[|\bX_{\scT}^x(t)-x^*|^2 + |\bu_{\scT}^x(t)-u^*|^2\]dt\rt\}^{1/2}.  
\end{align*}
Thus, the integral turnpike property \rf{Property:ITP} implies   
$$
x^* = \lim_{T\to\i}\frac{1}{T}\int_0^T \bX_{\scT}^x(t) dt, \q 
u^* = \lim_{T\to\i}\frac{1}{T}\int_0^T \bu_{\scT}^x(t) dt.
$$
Since $f$ is convex on $\dbR^n\times\dbR^m$, it is continuous. Applying Jensen's inequality, we obtain 
$$
f(x^*,u^*) 
=\lim_{T\to\i}f\lt(\frac{1}{T}\int_0^T\bX_{\scT}^x(t)dt,\frac{1}{T}\int_0^T\bu_{\scT}^x(t)dt\rt)
\les\liminf_{T\to\i}\frac{1}{T}\int_0^T f(\bX_{\scT}^x(t),\bu_{\scT}^x(t)) dt. 
$$

{\it Step 2.} We claim that for any $(z,v)\in\sV_x$, we can find a bounded continuous function 
$\ti{u}:[0,\i)\to\dbR^m$ such that the solution $\ti{X}(\cd)$ of 
\begin{equation}\label{eq:tX}\lt\{\begin{aligned}
\dot X(t) &= AX(t) + B\ti{u}(t) +b, \q t\ges0, \\
     X(0) &= x 
\end{aligned}\rt.\end{equation}
is bounded and  
\begin{equation}\label{IntTp1}
\int_0^\i \[|\ti{X}(t)-z|^2 + |\ti{u}(t)-v|^2\]dt <\i. 
\end{equation}

To prove the claim, recall the matrices $P_1$ and $P_2$ defined in \rf{Def-P12}.
Fix $(z,v)\in\sV_x$ and define the feedback control
$$
\ti{u}(t) \deq FP_1^{\top}[\ti{X}(t)-z] + v, \q t\ges 0,
$$
where $F\in\dbR^{m\times k}$ is to be specified, and $\ti{X}(\cd)$ is the solution of \rf{eq:tX} 
corresponding to the above $\ti{u}(\cd)$. For $i=1,2$, set 
$$Y_i(t) \deq P_i^{\top}[\ti{X}(t)-z], \q t\ges 0.$$
Then, by the controllable decomposition \rf{Decomposition-AB} and the fact that 
$$
(z,v)\in\sV_x \q\Longrightarrow\q Az+ Bv + b=0,
$$
the pair $(Y_1(\cd),Y_2(\cd))$ satisfies
$$\lt\{\begin{aligned}
\dot Y_1(t) &= (A_{11}+B_1F)Y_1(t) + A_{12}Y_2(t), \q t\ges0, \\
\dot Y_2(t) &= A_{22}Y_2(t), \q t\ges0, \\
Y_1(0) & = P_1^{\top}(x-z), \q Y_2(0)=P_2^{\top}(x-z). 
\end{aligned}\rt.$$
Since $x\in\sX$ and 
$$
(z,v)\in\sV_x \q\Longrightarrow\q P_2^\top z+c=Q_2(P_2^\top x+c),
$$ 
it follows that   
$$
P_2^{\top}(x-z)=Q_1(P_2^{\top}x+c) + Q_2(P_2^{\top}x+c) - (P_2^{\top}z+c) = Q_1(P_2^{\top}x+c)\in\sG_1.
$$
Thus, there exist constants $K_1,\l>0$ such that 
$$
|Y_2(t)| = |e^{tA_{22}}Q_1(P_2^{\top}x+c)| \les K_1(1+|x|)e^{-\l t}, \q\forall t\ges 0. 
$$
Since the pair $(A_{11},B_1)$ is controllable, we may invoke the pole-placement theorem 
to choose an $F$ such that 
$$
|e^{t(A_{11}+B_1F)}|\les K_2 e^{-2\l t}, \q\forall t\ges 0 
$$
for some constants $K_2>0$. Consequently,  
\begin{align*}
|Y_1(t)| & \les |e^{t(A_{11}+B_1F)}P_1^{\top}(x-z)| + \int_0^t|e^{(t-s)(A_{11}+B_1F)}|\cd|A_{12}|\cd|Y_2(s)|ds \\
&\les K_2 e^{-2\l t}(|x|+|z|) + K_1K_2|A_{12}|e^{-2\l t}(1+|x|)\int_{0}^{t}e^{\l s}ds \\
&\les K_3(1+|x|+|z|) e^{-\l t}, \q\forall t\ges 0, 
\end{align*}
where $K_3\deq K_2+\frac{K_1K_2|A_{12}|}{\l}$. It then follows from 
$$
\ti{X}(t) = P_1Y_1(t) + P_2Y_2(t) + z, \q \ti{u}(t) = FP_1^{\top}[\ti{X}(t)-z] + v, 
$$ 
that both $\ti{X}(\cd)$ and $\ti{u}(\cd)$ are bounded and 
\begin{align*}
\int_0^\i\[|\ti{X}(t)-z|^2 + |\ti{u}(t)-v|^2\]dt 
&\les \bigl(1+|FP_1^{\top}|^2\bigr)\int_0^\i|\ti{X}(t)-z|^2]dt \\
&= \bigl(1+|FP_1^{\top}|^2\bigr)\int_0^\i\[|Y_1(t)|^2 + |Y_2(t)|^2\]dt \\
&<\i.
\end{align*}

{\it Step 3.}  We now prove that   
\begin{equation}\label{upper-bound}
\limsup_{T\to\i}\frac{1}{T}\int_0^Tf(\bX_{\scT}^x(t),\bu_{\scT}^x(t))dt \les f(z,v), \q\forall (z,v)\in\sV_{x}.
\end{equation} 
To this end, fix $(z,v)\in\sV_x$ and let $(\ti{X}(\cd),\ti{u}(\cd))$ be constructed as in Step 2.
Since $(\bX_{\scT}^x(\cd),\bu_{\scT}^x(\cd))$ is an optimal pair, we have  
$$
\int_0^Tf(\bX_{\scT}^x(t),\bu_{\scT}^x(t))dt \les \int_0^Tf(\ti{X}(t),\ti{u}(t))dt.
$$
Thus, to prove \rf{upper-bound}, it sufficient to show that 
\begin{equation}\label{X1u1-limit}
\lim_{T\to\i}\frac{1}{T}\int_0^Tf(\ti{X}(t),\ti{u}(t))dt=f(z,v).
\end{equation}
Since both $\ti{X}(\cd)$ and $\ti{u}(\cd)$ are bounded, we can choose a compact set 
$\Om\subseteq\dbR^{n}\times\dbR^{m}$ such that $(z,v)\in\Om$ and $(\ti{X}(t),\ti{u}(t))\in\Om$ 
for all $t\ges0$. Set 
$$ 
L\deq\max_{(\xi,\eta)\in\Om}|f(\xi,\eta)|.
$$
Since $f$ is uniformly continuous on $\Om$, for any $\e>0$, we can find a $\d>0$ such that for all $(\xi_1,\eta_1),(\xi_2,\eta_2)\in\Om$,
$$
|\xi_1-\xi_2|^2+|\eta_1-\eta_2|^2<\d \q \Rightarrow \q |f(\xi_1,\eta_1)-f(\xi_2,\eta_2)|<\e.
$$
Now define
$$
E_{\scT}^{\d}\deq \{t\in[0,T]: |\ti{X}(t)-z|^2 + |\ti{u}(t)-v|^2\ges \d\}.
$$
By Markov's inequality and \rf{IntTp1},  
$$
\mu(E_{\scT}^{\d})\les \frac{1}{\d}\int_0^T\[|\ti{X}(t)-z|^2 + |\ti{u}(t)-v|^2\]dt \les \frac{1}{\d}\int_0^\i\[|\ti{X}(t)-z|^2 + |\ti{u}(t)-v|^2\]dt,
$$
where $\mu$ denotes the Lebesgue measure on $\dbR$. Consequently, 
\begin{align*}
&\frac{1}{T}\int_0^T|f(\ti{X}(t),\ti{u}(t))-f(z,v)|dt \\
&\q= \frac{1}{T}\int_{[0,T]\setminus E_{\scT}^{\d}}|f(\ti{X}(t),\ti{u}(t))-f(z,v)|dt 
    +\frac{1}{T}\int_{E_{\scT}^{\d}}|f(\ti{X}(t),\ti{u}(t))-f(z,v)|dt \\
&\q\les \frac{\e}{T}\[T-\mu(E_{\scT}^{\d})\] + \frac{2L\mu(E_{\scT}^{\d})}{T}\to \e, \q\hb{as } T\to\i. 
\end{align*}
Sending $\e\to0$ completes the proof. 

\ms

{\it Step 4.} Combining \rf{lower-bound} and \rf{upper-bound}, we obtain \rf{eq:f*<f}. 
\end{proof}

\section{The exponential turnpike property}\label{Sec:PETP} 

In \autoref{Sec:NC-TP}, we showed that, for the integral turnpike property
\rf{Property:ITP} to hold at an initial state $x$, it is necessary that
conditions \ref{A1} and \rf{Sec3:NCx0} be satisfied. 
In \autoref{sec:steady-pair}, we further established that, for any initial
state $x$ satisfying \rf{Sec3:NCx0}, the associated steady pair $(x^*,u^*)$
solves the constrained optimization Problem (O) corresponding to this initial
state $x$.

\ms 

In this section, we show that, under suitable assumptions on the stage cost
function $f$, both Problem (LC)$_{\scT}$ and Problem (O) admit unique solutions. 
Moreover, for any feasible initial state $x\in\sX$ (recalling the definition 
\rf{def:feasible-x} of $\sX$), we establish the exponential 
turnpike property \rf{Property:ETP} and show that the optimal value $V_{\scT}(x)$  
of Problem (LC)$_{\scT}$ and the minimum value $V^*(x)$ of Problem (O) satisfy 
the following estimate:   
\begin{equation}\label{VT-V*}
\frac{1}{T}V_{\scT}(x)-V^*(x) = O\(\frac{1}{T}\),
\end{equation} 
which characterizes the convergence rate of the averaged finite-horizon optimal
cost toward the steady optimal value.   

\ms

We impose the following assumptions on the stage cost function $f$.
\begin{taggedassumption}{(A2)}\label{A2} 
The stage cost function $f\in C^2(\dbR^n\times\dbR^m)$ is strongly convex, 
that is, there exists a constant $\d>0$ such that
$$
\nabla^2 f(x,u)\ges \d I_{n+m}, \q\forall (x,u)\in\dbR^n\times\dbR^m. 
$$
\end{taggedassumption}

The following result establish the unique solvability of Problem (LC)$_{\scT}$.

\begin{theorem}\label{thm:LC-solvable}
Let {\rm\ref{A2}} hold. Then for any $T>0$, Problem (LC)$_{\scT}$ admits a unique optimal 
control in $\sU[0,T]\deq L^2(0,T;\dbR^m)$ for every initial state $x$. 
\end{theorem}

\begin{proof}
By Taylor's formula,
\begin{equation}\label{Taylor-f}
f(x,u) \ges f(0,0) + f_x(0,0)^\top x + f_u(0,0)^\top u + \frac{\d}{2}\bigl(|x|^2+|u|^2\bigr), 
\q\forall (x,u)\in\dbR^n\times\dbR^m.  
\end{equation}
Using the elementary inequality $a^\top y\ges -\frac{1}{2\d}|a|^2-\frac{\d}{2}|y|^2$, we obtain
\begin{equation}\label{L-bound:f}
f(x,u) \ges f(0,0)-\frac{1}{2\d}\,|\nabla f(0,0)|^2, \q\forall (x,u)\in\dbR^n\times\dbR^m. 
\end{equation}
Hence $f$ is bounded below on $\dbR^n\times\dbR^m$. 
Without loss of generality, define
$$
\ti{f}(x,u)\deq f(x,u)-f(0,0)+\frac{1}{2\delta}\,|\nabla f(0,0)|^2 .
$$
Then $\ti{f}\ges0$. Since this only shifts the objective by a constant, we may
assume $f\ges0$.

\ms 

Let $\{u_k(\cd)\}\subseteq\sU[0,T]$ be a minimizing sequence such that 
\begin{equation}\label{Minimizing:uk}
\lim_{k\to\i}J_{\scT}(x;u_k(\cd)) = \inf_{u(\cd)\in\sU[0,T]} J_{\scT}(x;u(\cd)), 
\end{equation}
and let $X_k(\cd)$ denote the state trajectory corresponding to $u_k(\cd)$. 
By \rf{Taylor-f} and $f\ges0$, 
$$
f(x,u)\ges f_x(0,0)^\top x + f_u(0,0)^\top u + \frac{\d}{2}\bigl(|x|^2+|u|^2\bigr)
\ges {\d\over4}\bigl(|x|^2+|u|^2\bigr)-{1\over\d}|\nabla f(0,0)|^2.
$$
It then follows that
\begin{align*}
{\d\over4}\int_0^T|u_k(t)|^2dt \les {T\over\d}|\nabla f(0,0)|^2 + \int_0^T f(X_k(t),u_k(t))dt,
\end{align*}
which, together with \rf{Minimizing:uk}, implies that $\{u_k(\cd)\}$ is bounded 
in the Hilbert space $\sU[0,T]$. 
Hence, by weak compactness, it admits a weakly convergent subsequence. 
Without loss of generality, we may assume that $\{u_k(\cd)\}$ itself converges weakly 
to $\bu(\cd)\in\sU[0,T]$. 
By Mazur's theorem there exist $\a_{kj}\in[0,1]$, $j=1,2,...,N_k$, such that
\begin{equation*} 
\sum_{j=1}^{N_k}\a_{kj}=1, \q \lim_{k\to\i}\int_0^T\bigg|\sum_{j=1}^{N_k}\a_{kj}u_{k+j}(t)-\bu(t)\bigg|^2dt=0.
\end{equation*}
Define $\ti{u}_k(\cd)\deq\sum_{j=1}^{N_k}\a_{kj}u_{k+j}(\cd)$ and let $\ti{X}_k(\cd)$ 
be the corresponding state trajectory. Then 
$$
\sup_{0\les t\les T}|\ti{X}_k(t)-\bX(t)|\to0 \q\text{as } k\to\i,  
$$
where $\bX(\cd)$ denotes the state trajectory corresponding to $\bu(\cd)$. Consequently, 
$$
(\ti{X}_k(t),\ti{u}_k(t)) \to (\bX(t),\bu(t)) \q\ae~ t\in[0,T].
$$
Since $f\ges0$, it follows from Fatou's lemma and the convexity of $f$ that 
\begin{align*}
J_{\scT}(x;\bu(\cd)) 
&=\int_0^T f(\bX(t),\bu(t))dt \les \liminf_{k\to\i}\int_0^T f(\ti{X}_k(t),\ti{u}_k(t))dt \\
&\les \liminf_{k\to\i} \sum_{j=1}^{N_k}\a_{kj}\int_0^Tf(X_{k+j}(t),u_{k+j}(t))dt 
=\inf_{u(\cd)\in\sU[0,T]} J_{\scT}(x;u(\cd)). 
\end{align*}
This shows that $\bu(\cd)$ is an optimal control for the initial state $x$. 
The uniqueness follows from the strong (hence strict) convexity of $f$.   
\end{proof}

Recall the matrices $P_i$ and $Q_i$ ($i=1,2$) introduced in \rf{Def-P12} and \rf{def:Qi}, respectively.
The following result establishes the unique solvability of Problem (O).

\begin{theorem}\label{thm:O-solvable}
Let {\rm\ref{A1}--\ref{A2}} hold. Then for each $x\in\sX$, Problem (O) admits a unique solution. 
Moreover, $(x^*,u^*)$ is the solution of Problem (O) corresponding to $x$ if and only if 
\begin{equation}\label{Prob-O-xu}
\lt\{\begin{aligned}
&f_{x}(x^*,u^*) + A^{\top}\l_1^* + P_2\l_2^*=0,\\
&f_{u}(x^*,u^*) + B^{\top}\l_1^*=0,\\
&Ax^* + B u^* + b=0, \\
& P_2^{\top}x^* + c = Q_2(P_2^{\top}x + c),
\end{aligned}\rt.
\end{equation}
for some $\l_1^*\in\dbR^{n}$ and $\l_2^*\in\dbR^{n-k}$. 
\end{theorem}

\begin{proof}
Fix an $x\in\sX$. Since $\sV_{x}$ is a convex set and $f$ is strongly (hence strictly) convex,
the solution (if it exists) of Problem (O) corresponding to $x$ is unique. 

\ms 

Next, we show that if \rf{Prob-O-xu} holds for some $\l_1^*\in\dbR^n$ 
and $\l_2^*\in\dbR^{n-k}$, then $(x^*,u^*)$ is the unique solution to Problem (O) for this $x$.
Indeed, by Taylor's formula and \ref{A2}, we have
\begin{align*}
f(z,u)-f(x^*,u^*) 
&\ges \lan f_x(x^*,u^*), z-x^*\ran + \lan f_{u}(x^*,u^*),u-u^*\ran, \q\forall (z,u)\in\sV_{x}.
\end{align*}
On the other hand, \eqref{Prob-O-xu} implies 
\begin{align*}
&\lan f_x(x^*,u^*), z-x^*\ran + \lan f_{u}(x^*,u^*),u-u^*\ran \\
&\q= -\lan A^{\top}\l_1^* + P_2\l_2^*, z-x^* \ran - \lan B^{\top}\l_1^*, u-u^*\ran\\
&\q= -\lan \l_1^*, A(z-x^*)+B(u-u^*)\ran -\lan \l_2^*, P_2^{\top}(z-x^*)\ran \\
&\q= 0, \q\forall (z,u)\in\sV_{x},
\end{align*}
where in the last step, we have used the fact that 
$$
(x^*,u^*),(z,u)\in\sV_{x} \q\Longrightarrow\q  Az+Bu=Ax^*+Bu^*=-b, ~P_2^{\top}z=P_2^{\top}x^*=-c.
$$
Consequently, $ f(x^*,u^*)\les f(z,u)$ for all $(z,u)\in\sV_{x}$.

\ms

Now we show that equation \rf{Prob-O-xu} admits a solution $(x^*,u^*,\l_1^*,\l_2^*)$.  
Recall the notation introduced in \rf{Controllable--set}--\rf{Decomposition-AB}.
For $i=1,2$, let $\l_{1i}^*\deq P_i^{\top}\l_1^*$ and $y_i^*\deq P_i^{\top}x^*$. Then
$$
x^* = P_1y_1^*+P_2y_2^*.
$$
Premultiplying the first equation in \eqref{Prob-O-xu} by $P^{\top}$, we obtain
\begin{gather*}
P_1^{\top}f_x(x^*,u^*) + A_{11}^{\top}\l_{11}^* =0, \\
P_2^{\top}f_x(x^*,u^*) + A_{12}^{\top}\l_{11}^* + A_{22}^{\top}\l_{12}^* + \l_2^*=0.
\end{gather*}
The second equation in \eqref{Prob-O-xu} becomes
$$
0 = f_{u}(x^*,u^*) + B^{\top}PP^{\top}\l_1^* =f_{u}(x^*,u^*) + B_1^{\top}\l_{11}^*.
$$
Premultiplying the third equation in \eqref{Prob-O-xu} by $P^{\top}$, we obtain
$$
A_{11}y_1^* + A_{12}y_2^* + B_1 u^* + P_1^{\top}b=0, \q  A_{22}y_2^* + P_2^{\top}b =0. 
$$
Since $P^{\top}$ is invertible, we see that \rf{Prob-O-xu} is equivalent to the following: 
\begin{equation}\label{26-1-14:xu_star}\lt\{\begin{aligned}
&P_1^{\top}f_x(P_1y_1^* + P_2y_2^*,u^*) + A_{11}^{\top}\l_{11}^* =0, \\
&f_u(P_1y_1^* + P_2y_2^*,u^*) + B_1^{\top}\l_{11}^*=0, \\
&A_{11}y_1^* + B_1u^* + P_1^{\top}b + A_{12}y_2^*=0,\\
&\l_2^*=-[P_2^{\top}f_x(P_1y_1^* + P_2y_2^*,u^*) + A_{12}^{\top}\l_{11}^* + A_{22}^{\top}\l_{12}^*],\\
&y_2^* = Q_2(P_2^{\top}x+c)-c.
\end{aligned}\rt.\end{equation}
Since for any $y_1^*$, $y_2^*$, $u^*$, and $\l_{11}^*$, one can always choose $\l_2^*$ and $\l_{12}^*$ 
so that the fourth equation is satisfied, and since $y_2^*$ is uniquely determined by the last equality, 
it suffices to verify the existence of the triple $(y_1^*,u^*,\l_{11}^*)$ satisfying the first three equations. 
To this end, let     
\begin{align*}
g(z,u,\nu) \deq 
\bp P_1^{\top}f_{x}(P_1 z + P_2y_2^*,u) + A_{11}^{\top}\nu \\ 
    f_{u}(P_1 z + P_2y_2^*,u) + B_1^{\top}\nu \\
    A_{11}z + B_1 u
\ep, \q 
d\deq \bp 0 \\ 0 \\ -P_1^{\top}b-A_{12}y_2^*\ep.
\end{align*}
By assumption \ref{A2}, the mapping $g:\dbR^{m+2k}\to\dbR^{m+2k}$ is continuously differentiable, 
with Jacobian matrix 
$$
H(z,u,\nu) \deq \bp Q(z,u) & D^{\top} \\ D & 0 \ep,
$$
where $Q(z,u)\deq\diag\!(P_1^\top,I_m)\nabla^2f(P_1z+P_2y_2^*,u)\diag\!(P_1,I_m)$ and $D\deq(A_{11},B_1)$.
By \ref{A2},
$$
Q(z,u)\ges \d\diag(P_1^\top,I_m)\diag(P_1,I_m)=\d I_{k+m}.
$$
Thus, for any $(\xi,\eta)\in\ker\!H(z,u,\nu)$, we have
$$
Q(z,u)\xi + D^{\top}\eta=0, ~D\xi=0 \q\Longrightarrow\q  
DQ(z,u)^{-1}D^{\top}\eta=0          \q\Longrightarrow\q 
D^{\top}\eta=0.
$$
Since the pair $(A_{11},B_1)$ is controllable, the Hautus test implies that $\eta=0$. 
Consequently, $\xi=0$ as well.
Therefore, the Jacobian matrix $H(z,u,\nu)$ is invertible for all $(z,u,\nu)$. 
If we can prove 
\begin{equation}\label{g:coercive}
|(z,u,\nu)|\to\i \q\Longrightarrow\q |g(z,u,\nu)|\to\i,
\end{equation}
then, by Hadamard’s global inverse function theorem, there exists a unique triple $(y_1^*,u^*,\l_{11}^*)$ 
such that $g(y_1^*,u^*,\l_{11}^*)=d$. To verify \rf{g:coercive}, set $\xi\deq(z,u)$ and define
$$
\F(z,u) \deq \bp P_1^{\top}f_x(P_1 z+P_2 y_2^*,u) \\ f_u(P_1 z+P_2 y_2^*,u)\ep, \q
\bar Q(z,u) \deq \int_0^1 Q(tz,tu)dt \ges \d I_{k+m}.
$$
By Taylor's formula with integral remainder applied at $(0,0)$, we have
$$
\F(z,u)-\F(0,0) = \int_0^1 Q(tz,tu)\xi dt \equiv \bar Q(z,u)\xi.
$$
Taking the inner product of $g(z,u,\nu)$ with $\bigl(\begin{smallmatrix}\xi\\-\nu\end{smallmatrix}\bigr)$
and recalling $D\deq(A_{11},B_1)$, we obtain
\begin{align*}
|g(z,u,\nu)|\sqrt{|\xi|^2+|\nu|^2}
&\ges\Blan g(z,u,\nu),\Bigl(\begin{smallmatrix}\xi\\[0.6mm] -\nu\end{smallmatrix}\Bigr)\Bran 
=\Blan\Bigl(\begin{smallmatrix}\F(z,u)+D^{\top}\nu\\[0.6mm] D\xi\end{smallmatrix}\Bigr),
      \Bigl(\begin{smallmatrix}\xi\\[0.6mm] -\nu\end{smallmatrix}\Bigr)\Bran 
=\lan\F(z,u),\xi\ran \\
&= \lan\F(0,0),\xi\ran + \lan\bar Q(z,u)\xi,\xi\ran \ges \d|\xi|^2 - |\F(0,0)|\cd|\xi|,
\end{align*}
which implies
$$
|g(z,u,\nu)|\ges\frac{\d|\xi|^2-|\F(0,0)|\cd|\xi|}{\sqrt{|\xi|^2+|\nu|^2}} \to \i, \q\text{as }|\xi|\to\i.
$$
Next, suppose that $|\xi|$ is bounded while $|\nu|\to\i$. 
Since the pair $(A_{11},B_1)$ is controllable, we have $\ker\!D^{\top}=\{0\}$. 
Hence, there exists a constant $K>0$ such that
$$
|D^{\top}\nu|\ges K|\nu|, \q\forall \nu\in\dbR^k.
$$
Because $\F$ is continuous, $\F(z,u)$ remains bounded when $\xi$ is bounded. Therefore,  
$$
|g(z,u,\nu)|\ges|\F(z,u)+D^{\top}\nu|\ges|D^{\top}\nu|-|\F(z,u)|\to\i, \q\text{as }|\nu|\to\i.  
$$
Combining the above two cases, we obtain \rf{g:coercive}. 
\end{proof}

Next, we derive the estimate \rf{VT-V*}. In preparation, we present the following lemma.

\begin{lemma}\label{lmm:Lower-bound}
Let {\rm\ref{A1}--\ref{A2}} hold. Fix any $x\in\sX$, and let $(x^*,u^*,\l_1^*,\l_2^*)$ 
be a solution of equation \rf{Prob-O-xu}. 
Let $u(\cd)\in\sU[0,T]$ and let $X(\cd)$ be the corresponding state trajectory 
with the initial state $x$. Set
$$
\l_{1i}^*\deq P_i^{\top}\l_1^*, \q y_i^*\deq P_i^{\top}x^*, \q Y_i(t)\deq P_i^{\top}X(t), \q i=1,2.
$$
Then for any $0\les t_1<t_2\les T$,
\begin{align*}
&\int_{t_1}^{t_2}\[f(X(s),u(s))-V^*(x)+\lan A_{22}^\top\l_{12}^*+\l_2^*,Y_2(s)-y_2^*\ran\]ds \\
&\q=\lan\l_{11}^*,Y_1(t_1)-Y_1(t_2)\ran +\int_{t_1}^{t_2}\Blan\varPi(s)
    \Bigl(\begin{smallmatrix}X(s)-x^*\\[0.6mm] u(s)-u^*\end{smallmatrix}\Bigr), 
    \Bigl(\begin{smallmatrix}X(s)-x^*\\[0.6mm] u(s)-u^*\end{smallmatrix}\Bigr)\Bran ds, 
\end{align*}
where   
\begin{align}\label{def:varPi}
\varPi(s)\deq\int_0^1(1-\th)\nabla^2f\big((1-\th)x^*+\th X(s),(1-\th)u^*+\th u(s)\big)d\th
         \ges\frac{\d}{2}I_{n+m}.   
\end{align}
\end{lemma}

\begin{proof}
By the second-order Taylor formula with integral remainder, we have 
\begin{align}\label{second-Taylor}
f(X(s),u(s))-f(x^*,u^*) 
&= \lan f_{x}(x^*,u^*), X(s)-x^*\ran + \lan f_u(x^*,u^*), u(s)-u^*\ran \nn\\
&\hp{=\ } + \Blan\varPi(s)\Bigl(\begin{smallmatrix}X(s)-x^*\\[0.6mm] u(s)-u^*\end{smallmatrix}\Bigr), 
                          \Bigl(\begin{smallmatrix}X(s)-x^*\\[0.6mm] u(s)-u^*\end{smallmatrix}\Bigr)\Bran.  
\end{align}
Since $P=(P_1,P_2)$ is orthogonal, using \rf{Prob-O-xu} and \rf{Decomposition-AB} we obtain 
\begin{align*}
\lan f_x(x^*,u^*),X(s)-x^*\ran 
=& -\lan A^{\top}\l_1^*+P_2\l_2^*, X(s)-x^*\ran \\
=& -\lan P^\top(A^{\top}PP^\top\l_1^*+P_2\l_2^*),P^\top[X(s)-x^*]\ran \\
=& -\lan A_{11}^\top\l_{11}^*,Y_1(s)-y_1^*\ran 
   -\lan A_{12}^\top\l_{11}^*+A_{22}^\top\l_{12}^*+\l_2^*,Y_2(s)-y_2^*\ran\\
=& -\lan \l_{11}^*, A_{11}[Y_1(s)-y_1^*]+A_{12}[Y_2(s)-y_2^*]\ran \\
 & -\lan A_{22}^\top\l_{12}^*+\l_2^*,Y_2(s)-y_2^*\ran, \\[1mm]
\lan f_u(x^*,u^*),u(s)-u^*\ran 
=& -\lan \l_1^*, B[u(s)-u^*]\ran = -\lan P^{\top}\l_1^*, P^{\top}B[u(s)-u^*]\ran \\
=& -\lan \l_{11}^*, B_1[u(s)-u^*]\ran.  
\end{align*}
Combining the above equalities and noting that
\begin{equation}\label{Y1-Y2}\lt\{\begin{aligned}
{d\over ds}[Y_1(s)-y_1^*] &= A_{11}[Y_1(s)-y_1^*]+A_{12}[Y_2(s)-y_2^*] + B_1[u(s)-u^*],  \\
{d\over ds}[Y_2(s)-y_2^*] &= A_{22}[Y_2(s)-y_2^*],   
\end{aligned}\rt.\end{equation}
we obtain   
\begin{align*}
&\int_{t_1}^{t_2}\[\lan f_{x}(x^*,u^*), X(s)-x^*\ran + \lan f_u(x^*,u^*), u(s)-u^*\ran\]ds\\
&\q=\lan\l_{11}^*,Y_1(t_1)-Y_1(t_2)\ran-\int_{t_1}^{t_2}\lan A_{22}^\top\l_{12}^*+\l_2^*,Y_2(s)-y_2^*\ran ds, 
\end{align*}
which, together with \rf{second-Taylor}, yields the desired result.   
\end{proof}

We now establish the estimate \rf{VT-V*}. 

\begin{theorem}\label{thm:Value-function}
Let {\rm\ref{A1}--\ref{A2}} hold. Then for any $x\in\sX$, there exists a constant $K>0$, independent of $T$, such that 
\begin{equation}\label{Value-bound}
|V_{\scT}(x)-TV^*(x)|\les K,\q\forall T>0.
\end{equation}
\end{theorem}

\begin{proof}
For simplicity, in what follows we let $K>0$ denote a generic constant,
which may vary from line to line. 
By Step 2 in the proof of \autoref{prop:ProbO}, we can choose a bounded continuous function 
$u:[0,\i)\to\dbR^m$ such that the corresponding  state trajectory $X(\cd)$ 
is bounded and  
$$
L\deq \int_0^\i \[|X(t)-x^*|^2 + |u(t)-u^*|^2\]dt <\i. 
$$
With this control $u(\cd)$, the function $\varPi(\cd)$ defined by \rf{def:varPi} is bounded. Set
$$
Y(t) = \Bigl(\begin{smallmatrix}Y_1(t)\\[0.6mm]Y_2(t)\end{smallmatrix}\Bigr)\deq P^{\top}X(t), \q
 y^* = \Bigl(\begin{smallmatrix} y_1^*\\[0.6mm]y_2^* \end{smallmatrix}\Bigr)\deq P^{\top}x^*.
$$
Then $Y(\cd)$ and $y^*$ satisfy equation \rf{Y1-Y2}, and by \autoref{lmm:Lower-bound} we have
(noting that $Y(\cd)$ is bounded on $[0,\i)$)
\begin{align}\label{26-1-8:ineq1}
V_{\scT}(x)-TV^*(x) 
&\les \int_{0}^{T}\[f(X(t),u(t))-V^*(x)\]dt \nn\\
&\les \lan\l_{11}^*, Y_1(0)-Y_1(T)\ran-\int_0^T\lan A_{22}^\top\l_{12}^*+\l_2^*,Y_2(s)-y_2^*\ran ds \nn\\ 
&\hp{=\ } + K\int_{0}^{T}\[|X(t)-x^*|^2 + |u(t)-u^*|^2\] dt \nn\\
&\les K+K\int_0^T|Y_2(s)-y_2^*|ds +KL, \q\forall T>0.
\end{align}
Since $x\in\sX$, it follows from \autoref{prop:steady-1} that 
$$
Y_2(0)-y_2^*=P_2^{\top}(x-x^*)=P_2^{\top}x+c - Q_2(P_2^{\top}x+c)=Q_1(P_2^{\top}x+c)\in\sG_1.
$$
Thus, there exist constants $K,\rho>0$, independent of $T$, such that 
$$
|Y_2(t)-y_2^*| \les K|x-x^*|e^{-\rho t}, \q\forall 0\les t\les T.
$$
It follows that 
\begin{align}\label{26-1-8:ineq2}
\int_0^T|Y_2(s)-y_2^*|ds \les {K\over\rho}|x-x^*|, \q\forall T>0.
\end{align}
Combining \rf{26-1-8:ineq1} and \rf{26-1-8:ineq2} yields
$$
V_{\scT}(x)-TV^*(x) \les K, \q\forall T>0.
$$
For the lower bound, let $(\bX^x_{\scT}(\cd),\bu^x_{\scT}(\cd))$ denote the optimal pair of 
Problem (LC)$_{\scT}$ associated with the initial state $x$. 
Noting that by \ref{A2}, $\varPi(\cd)\ges\frac{\d}{2}I_{n+m}$, 
and again invoking \autoref{lmm:Lower-bound}, we obtain  
\begin{align*}
V_{\scT}(x)-TV^*(x) 
&\ges \frac{\d}{2}\int_0^T\[|\bX^x_{\scT}(t)-x^*|^2+|\bu^x_{\scT}(t)-u^*|^2\]dt \\
&\hp{=\ } -K|\bY_1(0)-\bY_1(T)|-K\int_0^T|\bY_2(s)-y_2^*|ds,
\end{align*}
where $\bY(t)=\Bigl(\begin{smallmatrix}\bY_1(t)\\\bY_2(t)\end{smallmatrix}\Bigr)\deq P^{\top}\bX^x_{\scT}(t)$
satisfies equation \rf{Y1-Y2}.
Similarly to \rf{26-1-8:ineq2}, one can show that $\int_0^T|\bY_2(s)-y_2^*|ds$ is bounded in $T$. 
Moreover, noting that $|\bY_1(0)-\bY_1(T)|\les|\bY(0)-\bY(T)|$, we obtain that for any $\e>0$,
\begin{align}\label{12-19:1}
&V_{\scT}(x)-TV^*(x) \nn\\
&\q\ges \frac{\d}{2}\int_0^T\[|\bX^x_{\scT}(t)-x^*|^2+|\bu^x_{\scT}(t)-u^*|^2\]dt 
      -K|\bar Y(0)-\bar Y(T)| - K, \nn\\
&\q\ges \frac{\d}{2}\int_0^T\[|\bX^x_{\scT}(t)-x^*|^2+|\bu^x_{\scT}(t)-u^*|^2\]dt 
      -\e|\bar Y(0)-\bar Y(T)|^2 - \frac{K^2}{4\e}-K.   
\end{align}
Now applying \autoref{lmm:X-bound} to \rf{Y1-Y2} and noting $|\bY(t)-y^*|=|\bX^x_{\scT}(t)-x^*|$, we obtain
\begin{align}\label{12-19:2}
\sup_{0\les t\les T}|\bX^x_{\scT}(0)-\bX^x_{\scT}(t)|^2 
&=\sup_{0\les t\les T}|\bar Y(0)-\bar Y(t)|^2 \les 2|\bY(0)-y^*|^2+2\sup_{0\les t\les T}|\bY(t)-y^*|^2 \nn\\
&\les 4|\bY(0)-y^*|^2 + K\int_{0}^{T}\[|\bY(t)-y^*|^2 + |\bu^x_{\scT}(t)-u^*|^2\] dt \nn\\
&\les 4|x-x^*|^2 + K\int_{0}^{T}\[|\bX^x_{\scT}(t)-x^*|^2 + |\bu^x_{\scT}(t)-u^*|^2\] dt.  
\end{align}
Substituting this estimate into the previous inequality and choosing
$\e>0$ sufficiently small, we obtain
$$
V_{\scT}(x)-T V^*(x)\ges-K, \q\forall T>0,
$$
where $K>0$ is a constant independent of $T$.  
\end{proof}

From the proof of the above theorem, we also obtain the following corollary. 

\begin{corollary}\label{crllry:MS-Tp}
Let {\rm\ref{A1}--\ref{A2}} hold. 
Then for any $x\in\sX$,  there exists a constant $K>0$, independent of $T$, such that 
$$
\sup_{0\les t\les T}|\bX^x_{\scT}(t)|^2 + \int_0^T\[|\bX^x_{\scT}(t)-x^*|^2+|\bu^x_{\scT}(t)-u^*|^2\]dt
\les K, \q\forall T>0, 
$$
where $(\bX^x_{\scT}(\cd),\bu^x_{\scT}(\cd))$ is the optimal pair of Problem (LC)$_{\scT}$ 
associated with $x$, and $(x^*,u^*)$ is the unique solution of Problem (O) corresponding to $x$.
\end{corollary}

\begin{proof}
Substituting \rf{12-19:2} into \rf{12-19:1} yields
$$
V_{\scT}(x)-T V^*(x)\ges \(\frac{\d}{2}-K\e\)\!\int_0^T\!\[|\bX^x_{\scT}(t)-x^*|^2+|\bu^x_{\scT}(t)-u^*|^2\]dt 
- 4\e|x-x^*|^2 - \frac{K^2}{4\e}-K.
$$
Choose $\e=\d/(4K)$ so that $\frac{\d}{2}-K\e=\d/4$. Then
\begin{equation}\label{MSbound:int}
\int_0^T\[|\bX^x_{\scT}(t)-x^*|^2+|\bu^x_{\scT}(t)-u^*|^2\]dt
\les K\[1+|x-x^*|^2\]+\frac{4}{\d}\[V_{\scT}(x)-TV^*(x)\].
\end{equation}
From \rf{12-19:2} and $\bX^x_{\scT}(0)=x$, we have 
\begin{equation}\label{MSbound:sup}
\sup_{0\les t\les T}|\bX^x_{\scT}(t)|^2
\les 2|x|^2+8|x-x^*|^2+K\int_0^T\[|\bX^x_{\scT}(t)-x^*|^2+|\bu^x_{\scT}(t)-u^*|^2\]dt.
\end{equation}
Substituting \rf{MSbound:int} into \rf{MSbound:sup} and using \rf{Value-bound},
we obtain the desired estimate. 
\end{proof}

For an initial state $x\in\sX$, let $(\bX^x_{\scT}(\cd),\bu^x_{\scT}(\cd))$ denote the optimal 
pair of Problem (LC)$_{\scT}$ associated with $x$, and let $(x^*,u^*)$ be the unique solution 
of Problem (O) corresponding to $x$.
We now turn to the establishment of the exponential turnpike property \rf{Property:ETP}. 
We first show that this property holds for the optimal state trajectory $\bX^x_{\scT}(\cd)$.

\begin{theorem}\label{thm:PETP-X}
Let {\rm\ref{A1}--\ref{A2}} hold. 
Then for any $x\in\sX$, there exist constants $K,\l>0$, independent of $T$, such that
\begin{equation}\label{26-01-05:PETP-X}
|\bX^x_{\scT}(t)-x^*|\les K\[e^{-\l t}+e^{-\l(T-t)}\], \q\forall 0\les t\les T.
\end{equation}
\end{theorem}  

\begin{proof}
Fix $x\in\sX$, and let $M$ be the minimal constant $K$ for which the estimate
in \autoref{crllry:MS-Tp} holds. 
For notational simplicity, in what follows we suppress the superscript $x$ in
$\bX^x_{\scT}(\cd)$ and $\bu^x_{\scT}(\cd)$.
Since $x\in\sX$, it follows from \autoref{prop:steady-1} that 
$$
P_2^{\top}(x-x^*)=P_2^{\top}x+c - Q_2(P_2^{\top}x+c)=Q_1(P_2^{\top}x+c)\in\sG_1.
$$
Thus, the solution $Z(\cd)$ to the ODE
$$
\dot Z(t) = A_{22}Z(t), \q Z(0)=P_2^{\top}(x-x^*) 
$$
satisfies
\begin{equation}\label{26-01-07:ineq1}
|Z(t)| \les K_1|x-x^*|e^{-\rho t}, \q\forall t\ges0,
\end{equation}
for some constants $K_1,\rho>0$. Choose $\e\in(0,1)$ and define
$$
\a_0\deq -\ln\e/\rho, \q T_0= M/\e^2.
$$
Then by \autoref{crllry:MS-Tp}, for any $T>2(T_0+\a_0)$, we have 
$$
\frac{1}{T_0}\int_{[\a_0,\a_0+T_0]\cup[T-\a_0-T_0,T-\a_0]}\[|\bX_{\scT}(t)-x^*|^2+|\bu_{\scT}(t)-u^*|^2\]dt 
\les\frac{M}{T_0} = \e^2.
$$
Consequently, by the integral mean value theorem, there exist $\a_1\in[\a_0,\a_0+T_0]$ 
and $\b_1\in[T-\a_0-T_0,T-\a_0]$ such that 
$$
|\bX_{\scT}(\a_1)-x^*|\les\e,  \q |\bX_{\scT}(\b_1)-x^*|\les\e.
$$
We claim that there exist constants $S,L>0$, independent of $T$, and two finite sequences 
$\{\a_i\}_{i=1}^{j}$ and $\{\b_i\}_{i=1}^{j}$ such that 
\begin{align}
\label{TpX-claim1}
& \a_0<\a_1<\cdots<\a_{j}<\b_{j}<\cdots<\b_1<T-\a_0,\q \b_j-\a_j<2S, \\
\label{TpX-claim2}
& \a_0\les \a_{i+1}-\a_{i}<S,\q \a_0\les \b_{i}-\b_{i+1}<S,\q i=1,\dots,j-1, \\
\label{TpX-claim3}
& |\bX_{\scT}(\a_i)-x^*|\les \e^i,\q |\bX_{\scT}(\b_i)-x^*|\les \e^i,\q i=1,\dots,j, \\
\label{Ineq-abl}
&\int_{\a_i}^{\b_i}\[|\bX_{\scT}(s)-x^*|^2+|\bu_{\scT}(s)-u^*|^2\]ds\les L\e^{2i}, \q i=1,\dots,j.
\end{align}
If we can establish this claim, then noting that 
$$
{d\over dt}[\bX_{\scT}(t)-x^*]=A[\bX_{\scT}(t)-x^*]+B[\bu_{\scT}(t)-u^*]
$$
and applying \autoref{lmm:X-bound}, we obtain from \rf{TpX-claim3} and \rf{Ineq-abl} that
\begin{align*}
\sup_{\a_i\les t\les\b_i}|\bX_{\scT}(t)-x^*|^2 
\les |\bX_{\scT}(\a_i)-x^*|^2+ K\!\int_{\a_i}^{\b_i}\!\[|\bX_{\scT}(s)-x^*|^2+|\bu_{\scT}(s)-u^*|^2\]ds  
\les K\e^{2i}.
\end{align*}
Here and throughout the remainder of the proof, $K>0$ denotes a generic
constant independent of $T$, whose value may vary from line to line.
It follows that
\begin{equation}\label{Sec4:Ineq_ab}
\begin{aligned}
&|\bX_{\scT}(t)-x^*|\les Ke^{i\ln\e}, \q t\in[\a_i,\a_{i+1})\cup(\b_{i+1},\b_i], \q i=1,\dots,j-1, \\
&|\bX_{\scT}(t)-x^*|\les Ke^{j\ln\e}, \q t\in[\a_j,\b_j].
\end{aligned}
\end{equation}
Moreover, by \rf{TpX-claim2} and noting that $\a_1\les\a_0+T_0$ and $\b_1\ges T-\a_0-T_0$, we have
\begin{align*}
&iS\ges \a_{i+1}-\a_1\ges t-\a_0-T_0,  \q\forall t\in[\a_i,\a_{i+1}], \q i=1,\dots,j-1, \\
&iS\ges \b_1-\b_{i+1}\ges T-t-a_0-T_0, \q\forall t\in[\b_{i+1},\b_i], \q i=1,\dots,j-1, \\
&jS=(j-1)S+S\ges\a_j-\a_1+S\ges t-\a_0-T_0,   \q\forall t\in[\a_j,\a_j+S],\\
&jS=(j-1)S+S\ges\b_1-\b_j+S\ges T-t-\a_0-T_0, \q\forall t\in[\b_j-S,\b_j]. 
\end{align*}
Set $\l\deq\frac{-\ln\e}{S}$. Combining the above estimates with \rf{Sec4:Ineq_ab} 
and noting that $\b_j-\a_j<2S$, we obtain
$$
|\bX_{\scT}(t)-x^*|\les Ke^{\l(\a_0+T_0)}\[e^{-\l t}  + e^{-\l(T-t)}\], \q\forall t\in[\a_1,\b_1].
$$
Since $\e$, $\a_0$, $T_0$, and $S$ are all independent of $T$, the constant
$Ke^{\l(\a_0+T_0)}$ is independent of $T$ as well. 
On the boundary layers $[0,\a_1]\cup[\b_1,T]$, the uniform bound
$$
\sup_{0\les t\les T}|\bX_{\scT}(t)|\les \sqrt{M}, \q\forall T>0
$$ 
from \autoref{crllry:MS-Tp} implies that $|\bX_{\scT}(t)-x^*|$ is uniformly bounded; 
hence, enlarging the prefactor if necessary, the same two-sided exponential bound holds 
for all $t\in[0,T]$, which yields \rf{26-01-05:PETP-X}.

\ms 

To complete the proof, it remains to establish the claim. 
We have already found $\a_0\les\a_1<\b_1\les T-\a_0$ such that 
\begin{gather*}
|\bX_{\scT}(\a_1)-x^*|\les\e,  \q |\bX_{\scT}(\b_1)-x^*|\les\e, \\
\int_{\a_1}^{\b_1}\[|\bX_{\scT}(s)-x^*|^2+|\bu_{\scT}(s)-u^*|^2\]ds\les M=T_0\e^{2}.
\end{gather*}
We next refine the estimate for the above integral term.
Recall the notation introduced in \rf{Def-P12} and \rf{Decomposition-AB}. Define  
$$
\bY_{\scT}(t)\deq P^\top[\bX_{\scT}(t)-x^*], \q \bY_{i,\scT}(t)\deq P_i^\top[\bX_{\scT}(t)-x^*], \q i=1,2.
$$
Then we have
$$\lt\{\begin{aligned}
\dot \bY_{1,\scT}(t) & = A_{11} \bY_{1,\scT}(t) + B_1[\bu_{\scT}(t)-u^*] + A_{12}\bY_{2,\scT}(t), \\
\dot \bY_{2,\scT}(t) & = A_{22} \bY_{2,\scT}(t), \\
\bY_{1,\scT}(0) & = P_1^{\top}(x-x^*),\q \bY_{2,\scT}(0)=P_2^{\top}(x-x^*). 
\end{aligned}\rt.$$
According to \rf{26-01-07:ineq1}, we have 
\begin{equation}\label{26-01-07:ineq1*}
|\bY_{2,\scT}(t)| = |Z(t)| \les K_1|x-x^*|e^{-\rho t}, \q\forall 0\les t\les T.
\end{equation}
Because the pair $(A_{11},B_1)$ is controllable, by the pole-placement theorem there exist 
a matrix $F\in\dbR^{m\times k}$ and a constant $K_2>0$ such that  
\begin{equation}\label{26-01-07:ineq2}
|e^{t(A_{11}+B_1F)}| \les K_2 e^{-2\rho t}, \q\forall t\ges0. 
\end{equation}
Define 
$$
\bv_{\scT}(t)\deq \bu_{\scT}(t)- u^* - F\bY_{1,\scT}(t), \q \hA_{11} \deq A_{11} + B_1F.
$$
Then $\bY_{1,\scT}(\cd)$ satisfies 
$$
\dot \bY_{1,\scT}(t) = \hA_{11} \bY_{1,\scT}(t) + B_1\bv_{\scT}(t) + A_{12}\bY_{2,\scT}(t),
\q   \bY_{1,\scT}(0) = P_1^{\top}(x-x^*).
$$ 
Let $(\bar\L_{\scT}(\cd),\bar\Si_{\scT}(\cd))$ be the solution to the following ODE:
$$\lt\{\begin{aligned}
\dot{\bar\L}_{\scT}(t) &= \hA_{11}{\bar\L}_{\scT}(t) + B_1\bv_{\scT}(t), \q t\in[0,T],\\
\dot{\bar\Si}_{\scT}(t) &= \hA_{11}{\bar\Si}_{\scT}(t) + A_{12}\bY_{2,\scT}(t), \q t\in[0,T],\\
\bar\L_{\scT}(0)& = P_1^{\top}(x-x^*),\q \bar\Si_{\scT}(0)=0.
\end{aligned}\rt.$$
Then 
\begin{align}\label{26-1-8:bY=}
\bY_{1,\scT}(t) = \bar\L_{\scT}(t) + \bar\Si_{\scT}(t),\q t\in[0,T].
\end{align}
Moreover, using \rf{26-01-07:ineq1*} and \rf{26-01-07:ineq2} we obtain
\begin{align}\label{26-1-8:Si}
|\bar\Si_{\scT}(t)|\les \int_{0}^{t}|e^{(t-s)\hA_{11}}| |A_{12}| |\bY_{2,\scT}(s)|ds 
\les \frac{K_1K_2|A_{12}| |x-x^*|}{\rho}e^{-\rho t},\q\forall t\in[0,T].
\end{align}
Now define 
\begin{align*}
\xi(t) &\deq -B_1^{\top}e^{-(t-\a_1)\hA_{11}^{\,\top}}
\lt(\int_0^1e^{-s\hA_{11}}B_1B_1^{\top}e^{-s\hA_{11}^{\,\top}}ds\rt)^{-1}\bar\L_{\scT}(\a_1), \q t\in[\a_1,\a_1+1], \\
\eta(t) &\deq  B_1^{\top}e^{-(t-\b_1)\hA_{11}^{\,\top}}
\lt(\int_{-1}^0e^{-s\hA_{11}}B_1B_1^{\top}e^{-s\hA_{11}^{\,\top}}ds\rt)^{-1}\bar\L_{\scT}(\b_1), \q t\in[\b_1-1,\b_1],
\end{align*}
and set
$$
v(t)\deq\begin{cases}
\bv_{\scT}(t), &\q t\in[0,\a_1),\\ 
\xi(t),        &\q t\in[\a_1,\a_1+1], \\
0,             &\q t\in (\a_1+1,\b_1-1),\\
\eta(t),       &\q t\in [\b_1-1,\b_1], \\
\bv_{\scT}(t), &\q t\in (\b_1,T].
\end{cases}
$$
Let $(Y_1(\cd),Y_2(\cd))$, $(\L(\cd),\Si(\cd))$, and $X(\cd)$ denote the solutions to 
$$
\lt\{\begin{aligned}
\dot Y_1(t) &= \hA_{11}Y_1(t) + B_1v(t) + A_{12}Y_2(t), \\
\dot Y_2(t) &= A_{22}Y_2(t), \\
Y_1(0) &= P_1^{\top}(x-x^*),\q Y_2(0)=P_2^{\top}(x-x^*),
\end{aligned}\rt. \q
\lt\{\begin{aligned}
\dot\L(t) &= \hA_{11}\L(t) + B_1v(t), \\
\dot\Si(t) &= \hA_{11}\Si(t) + A_{12}Y_2(t), \\
\L(0)& = P_1^{\top}(x-x^*),\q \Si(0)=0,
\end{aligned}\rt.
$$
and  
\begin{equation*}
\dot X(t) = AX(t) + Bu(t) +b, \q X(0) = x,
\end{equation*}
respectively, where $u(t)\deq v(t)+u^*+FY_1(t)$. Then  
\begin{equation}\label{Sec4:XbXt}\lt\{\begin{aligned}
Y_1(t)&= \L(t) + \Si(t),\q\forall t\in[0,T], \\
Y_2(t)&=\bY_{2,\scT}(t), \q\forall t\in[0,T],\\
\L(t)&=\bar\L_{\scT}(t), \q\forall t\in [0,\a_1]\cup[\b_1,T], \\ 
\L(t)&=0, \q\forall t\in [\a_1+1,\b_1-1],\\
\Si(t)&=\bar\Si_{\scT}(t),\q\forall t\in[0,T],\\
Y_i(t)&= P_i^\top[X(t)-x^*], \q\forall t\in[0,T], \q i=1,2.  
\end{aligned}\rt.\end{equation}
By the definitions of $\xi(\cd)$ and $\eta(\cd)$, we can find a constant $L_1>0$, 
depending only on $B_1$ and $\hA_{11}$, such that 
\begin{equation}\label{26-1-8:sup}\lt\{\begin{aligned}
&\sup_{t\in[\a_1,\a_1+1]}\[|\L(t)|^2+|v(t)|^2\]\les L_1^2|\bar\L_{\scT}(\a_1)|^2,\\ 
&\sup_{t\in[\b_1-1,\b_1]}\[|\L(t)|^2+|v(t)|^2\]\les L_1^2|\bar\L_{\scT}(\b_1)|^2.
\end{aligned}\rt.\end{equation}
From \rf{26-1-8:bY=} and \rf{26-1-8:Si}, by setting $K_3=\frac{K_1K_2|A_{12}||x-x^*|}{\rho}$ and 
noting that $\rho\b_1\ges\rho\a_1\ges\rho\a_0\ges-\ln\e$, we obtain
\begin{align*}
|\bar\L_{\scT}(\a_1)|
&\les |\bY_{1,\scT}(\a_1)|+|\bar\Si_{\scT}(\a_1)| 
\les |\bX_{\scT}(\a_1)-x^*| + K_3e^{-\rho\a_1} 
\les \e(1+K_3)
\les 1+K_3, \\
|\bar\L_{\scT}(\b_1)|
&\les |\bY_{1,\scT}(\b_1)|+|\bar\Si_{\scT}(\b_1)| 
\les |\bX_{\scT}(\b_1)-x^*| + K_3e^{-\rho\b_1} 
\les \e(1+K_3)
\les 1+K_3.  
\end{align*}
Consequently, by \rf{26-01-07:ineq1*} and \rf{26-1-8:Si}--\rf{26-1-8:sup},
\begin{align*}
|X(t)| &\les |x^*| + |X(t)-x^*| \les |x^*| + |Y_1(t)| + |Y_2(t)| \les |x^*| + |\L(t)| + |\Si(t)| + |Y_2(t)| \\
       &\les |x^*| + L_1(1+K_3) + K_3 + K_1|x-x^*| \equiv r_1, \q\forall t\in[\a_1,\b_1], \\
|u(t)| &\les |u^*| + |v(t)| + (|\L(t)| + |\Si(t)|)|F| \\
       &\les |u^*| + L_1(1+|F|)(1+K_3) + |F|K_3\equiv r_2, \q\forall t\in[\a_1,\b_1].
\end{align*} 
Define
$$
L_2 \deq \max_{|x|\les r_1, |u|\les r_2}|\nabla^2 f(x,u)|
$$
and note that by \eqref{Sec4:XbXt},
$$
Y_2(t)=\bY_{2,\scT}(t), ~\forall t\in[0,T]; \qq
Y_1(t)=\bY_{1,\scT}(t), ~\forall t\in [0,\a_1]\cup[\b_1,T].
$$
Then, by applying \autoref{lmm:Lower-bound} twice, we obtain  
\begin{align*}
&{\d\over2}\int_{\a_1}^{\b_1}\[|\bX_{\scT}(s)-x^*|^2+|\bu_{\scT}(s)-u^*|^2\]ds \\
&\q\les \int_{\a_1}^{\b_1}\[f(\bX_{\scT}(s),\bu_{\scT}(s))-V^*(x)
        +\lan\l_2^*+A_{22}^\top\l_{12}^*,\bY_{2,\scT}(s)\ran\]ds \\
&\hp{\q=\ } +\lan\l_{11}^*,\bY_{1,\scT}(\b_1)-\bY_{1,\scT}(\a_1)\ran \\
&\q\les \int_{\a_1}^{\b_1}\[f(X(s),u(s))-V^*(x)+\lan\l_2^*+A_{22}^\top\l_{12}^*,Y_2(s)\ran\]ds 
        +\lan\l_{11}^*,Y_1(\b_1)-Y_1(\a_1)\ran \\
&\q\les \frac{L_2}{2}\int_{\a_1}^{\b_1}\[|X(s)-x^*|^2+|u(s)-u^*|^2\]ds \\
&\q= \frac{L_2}{2}\int_{\a_1}^{\b_1}\[|Y(s)|^2+|v(s)+FY_1(s)|^2\]ds \\
&\q\les \frac{L_2}{2}\int_{\a_1}^{\b_1}\[2(1+2|F|^2)\(|\L(s)|^2+|\Si(s)|^2\)+|Y_2(s)|^2+2|v(s)|^2\]ds \\
&\q\les L_2(1+2|F|^2)\int_{[\a_1,\a_1+1]\cup[\b_1-1,\b_1]}\[|\L(s)|^2 + |v(s)|^2\]ds \\
&\hp{\q=\ }+ \frac{L_2}{2}\int_{\a_1}^{\b_1}\[2(1+2|F|^2)|\Si(s)|^2+|Y_2(s)|^2\]ds \\
&\q\les 2L_1^2L_2(1+2|F|^2)(1+K_3)^2\e^2 + \frac{L_2(1+2|F|^2)K_3^2}{2\rho}\e^2 
        + \frac{L_2K_1^2|x-x^*|^2}{4\rho}\e^2.        
\end{align*}
By setting 
\begin{equation}\label{def:constant-L}
L\deq \frac{2}{\d}\lt[2L_1^2L_2(1+2|F|^2)(1+K_3)^2 + \frac{L_2(1+2|F|^2)K_3^2}{2\rho} 
      + \frac{L_2K_1^2|x-x^*|^2}{4\rho}\rt],
\end{equation}
we obtain 
\begin{equation*}
\int_{\a_1}^{\b_1}\[|\bX_{\scT}(s)-x^*|^2+|\bu_{\scT}(s)-u^*|^2\]ds\les L\e^{2}.
\end{equation*}
Now, take $S\deq\a_0+L/\e^2$. Then we have 
\begin{align*}
&\frac{1}{S-\a_0}\int_{\a_1+\a_0}^{\a_1+S}\[|\bX_{\scT}(s)-x^*|^2+|\bu_{\scT}(s)-u^*|^2\]ds \les \e^{4},\\
&\frac{1}{S-\a_0}\int_{\b_1-S}^{\b_1-\a_0}\[|\bX_{\scT}(s)-x^*|^2+|\bu_{\scT}(s)-u^*|^2\]ds \les \e^{4}.
\end{align*}
By the integral mean value theorem, we can find $\a_2\in[\a_1+\a_0,\a_1+S)$ and 
$\b_2\in(\b_1-S,\b_1-\a_0]$ such that 
$$
|\bX_{\scT}(\a_{2})-x^*|\les\e^2,  \q |\bX_{\scT}(\b_{2})-x^*|\les\e^2.
$$
Repeating the previous procedure, we can show that
$$
\int_{\a_2}^{\b_2}\[|\bX_{\scT}(s)-x^*|^2+|\bu_{\scT}(s)-u^*|^2\]ds\les L\e^4
$$
with the same constant $L$ defined in \rf{def:constant-L}, and that there exist $\a_3\in[\a_2+\a_0,\a_2+S)$ 
and $\b_3\in(\b_2-S,\b_2-\a_0]$ such that 
$$
|\bX_{\scT}(\a_3)-x^*|\les\e^3,  \q |\bX_{\scT}(\b_3)-x^*|\les\e^3.
$$
Therefore, by induction, the claim holds.  
\end{proof}

\autoref{thm:PETP-X} establishes the exponential turnpike property for the optimal state trajectory 
$\bX^x_{\scT}(\cd)$ under the strong convexity condition \ref{A2} imposed on the stage cost function $f$.
However, condition \ref{A2} alone does not guarantee the exponential turnpike property for the optimal 
control $\bu^x_{\scT}(\cd)$ in general.
In the next result, we establish the exponential turnpike property for the optimal control 
$\bu^x_{\scT}(\cd)$ by imposing an additional assumption.
Let $\cB^n_r(x)$ denote the open ball of radius $r$ centered at $x\in\dbR^n$. 

\begin{taggedassumption}{(A3)}\label{A3}
For any $r>0$, there exists a constant $\g(r)>0$, such that 
$$
|f_{xu}(x,u)[f_{uu}(x,u)]^{-1}|\les \g(r),\q\forall(x,u)\in\overline{\cB_{r}^n(0)}\times \dbR^m.
$$
\end{taggedassumption}

\begin{theorem}\label{thm:PETP-u}
Let {\rm\ref{A1}--\ref{A3}} hold. 
Then for any $x\in\sX$, there exist constants $K,\l>0$, independent of $T$, such that 
\begin{equation}\label{26-1-15:PETP-u}
|\bu_{\scT}(t)-u^*|\les K\[e^{-\l t}+e^{-\l(T-t)}\],\q \ae~ t\in[0,T], \q\forall T>0.
\end{equation}
\end{theorem}


\begin{proof}
Fix $x\in\sX$, and let $M$ be the minimal constant $K$ for which the estimate
in \autoref{crllry:MS-Tp} holds.
We suppress the superscript $x$ in both $\bX_{\scT}^{x}(\cd)$ and $\bu_{\scT}^{x}(\cd)$.
By Pontryagin's maximum principle (see \cite{Pontryagin1962}), there exists an adjoint function $\bps(\cd)$ such that
\begin{equation}\label{Sec5:adj-psi}
\lt\{\begin{aligned}
&\dot{\bar\psi}_{\scT}(t) = -A^{\top}\bps(t)-f_{x}(\bX_{\scT}(t),\bu_{\scT}(t)), \q \ae~t\in[0,T],\\
&\bps(T)=0,
\end{aligned}\rt.
\end{equation}
and
\begin{equation}\label{26-1-14:stationary}
B^{\top}\bps(t) + f_{u}(\bX_{\scT}(t),\bu_{\scT}(t))=0,\q \ae~t\in[0,T].
\end{equation}
We divide the proof into the following steps.

\ms

{\it Step 1.} We show that $\bu_{\scT}(\cd)$ can be expressed as a continuously differentiable function of $(\bX_{\scT}(\cd),\bps(\cd))$ for almost every $t\in[0,T]$.
Under \ref{A2}, for each $(x,u)\in\dbR^{n}\times \dbR^m$, we obtain $f_{uu}(x,u)\ges \d I_{m}$, and hence $f_{uu}(x,u)$ is invertible. 
Moreover, the Taylor's formula with integral remainder yields
$$f_{u}(x,u)-f_u(x,0)=\int_{0}^{1}f_{uu}(x,\th u)ud\th \equiv Q(x,u)u.$$
Therefore, by the Cauchy–Schwarz inequality,
$$\q |f_u(x,u)-f_u(x,0)|\ges \dfrac{u^{\top}Q(x,u)u}{|u|}\ges \d |u|,\q \forall u\in \dbR^m\setminus\{0\},x\in\dbR^n.$$
Consequently, for each fixed $x$, the mapping $u \mapsto f_u(x,u)$ is coercive.
By Hadamard's global inverse function theorem, for every $(x,y)\in\dbR^{n}\times\dbR^{m}$, 
the equation $y+f_u(x,u)=0$ admits a unique solution $u=g(x,y)$.
Further, the implicit function theorem implies that $g$ is continuously differentiable and satisfies 
\begin{equation}\label{Sec5:def-g}
g_x(x,y) = -[f_{uu}(x,g(x,y))]^{-1}f_{ux}(x,g(x,y)),\q g_{y}(x,y)=-[f_{uu}(x,g(x,y))]^{-1}.
\end{equation}
As a result, \eqref{26-1-14:stationary} is equivalent to
\begin{equation}\label{26-1-14:g}
\bu_{\scT}(t)=g(\bX_{\scT}(t),B^{\top}\bps(t)),\q \ae~t\in[0,T].
\end{equation}

\ms

{\it Step 2.} We show that $|\bu_{\scT}(\cd)|$ is essentially uniformly bounded in $T$.
Set $h(x,y)\deq f_{x}(x,g(x,y))$ and note that 
\begin{align*}
h_{y}(x,y) = f_{xu}(x,g(x,y))g_{y}(x,y)=-f_{xu}(x,g(x,y))[f_{uu}(x,g(x,y))]^{-1}.
\end{align*}
Then by \ref{A3}, \autoref{crllry:MS-Tp} and setting $C_0 \deq \max_{|x|\les \sqrt{M}}|h(x,0)|+\g(\sqrt{M})$, we have 
\begin{equation}\label{Sec5:hxy}
\begin{aligned}
|f_x(\bX_{\scT}(t),\bu_{\scT}(t))|&=
|h(\bX_{\scT}(t), B^{\top}\bps(t))|\\
&\les |h(\bX_{\scT}(t),0)|+\int_{0}^{1}|h_y(\bX_{\scT}(t),\th B^{\top}\bps(t))|d\th \cd |B^{\top}\bps(t)| \\
& \les C_0\[1 + |B^{\top}\bps(t)|\], \q \ae~t\in[0,T].
\end{aligned}
\end{equation}
Recall the notation in \rf{Def-P12}, \rf{Decomposition-AB} and set 
$$\f_{\scT}(t) \deq P_1^{\top}\bps(t),\q t\in[0,T].$$ 
Together with \rf{Sec5:adj-psi}, we have for almost every $t\in[0,T]$ that 
\begin{equation}\label{26-1-13:eq-f}
\lt\{\begin{aligned}
\dot\f_{\scT}(t) &= -A_{11}^{\top}\f_{\scT}(t) - P_1^{\top}f_x(\bX_{\scT}(t),\bu_{\scT}(t)), \\
\f_{\scT}(T) & =0, \q B^{\top}\bps(t)=B_1^{\top}\f_{\scT}(t).  
\end{aligned}\rt.
\end{equation}

We now show the uniform boundedness of $\f_{\scT}(\cd)$.
First, for any $t\in[T-1,T]$, it follows by variation-of-constants formula  and \rf{Sec5:hxy} that 
\begin{align*}
|\f_{\scT}(t)|&\les \int_{t}^{T}|e^{(s-t)A_{11}^{\top}}P_1^{\top}|\cd|f_u(\bX_{\scT}(s),\bu_{\scT}(s))|ds \\
& \les K \int_{t}^{T}e^{\rho (s-t)}[1 + |\f_{\scT}(s)|]ds \les K + K \int_{t}^{T}e^{\rho (s-t)} |\f_{\scT}(s)|ds,
\end{align*}
where $\rho \deq \max\{\Re(\l)\nid \l\in \si(A_{11})\}+1$ 
and $K$ is a constant independent of $T$, and may vary form line to line.  
By using the Gronwall's inequality, we have 
$$|\f_{\scT}(t)|\les K\exp\Big\{\int_{t}^{T}e^{\rho(s-t)}ds\Big\}\les K\exp\Big\{\frac{e^{\rho}}{\rho}\Big\},\q \forall t\in[T-1,T],$$
which follows the uniform boundedness of $\f_{\scT}(\cd)$ on $[T-1,T]$.

\ms

Next, for $[0,T-1]$, we assume that $\p_{\scT}\deq \max_{t\in[0,T-1]}|\f_{\scT}(t)|$ is located at $\t$.  
Let $N_1$ be a constant to be determined and define
$$\lt\{\begin{aligned}
&N_2 \deq \max_{|x|\les \sqrt{M},|u-u^*|\les N_1}|f_{u}(x,u)|+1, \q  \Si_1 \deq \{t\in [\t,\t+1]\nid |\bu_{\scT}(t)-u^*|\ges N_1\}, \\  
&\Si_2 \deq \{t\in [\t,\t+1]\nid |B_1^{\top}\f_{\scT}(t)|\ges N_2\}. 
\end{aligned}\rt.$$
Then $\Si_2 \subseteq \Si_1$ and it follows from Markov's inequality  and \autoref{crllry:MS-Tp} that 
$$\mu(\Si_2) \les \mu(\Si_1) \les \frac{M}{N_1^2},$$
where $\mu(\cd)$ is the Lebesgue measure on $\dbR$.
Since $(A_{11}^{\top},B_1^{\top})$ is observable, 
by the observability inequality in \autoref{lmm:X-bound} (ii) and \rf{Sec5:hxy}--\rf{26-1-13:eq-f}, 
we can find a constant $C_1>0$, independent of $T$, such that
\begin{align*}
|\p_{\scT}|^2 & \les  C_1\int_{\t}^{\t+1}\[|B_1^{\top}\f_{\scT}(s)|^2 + |f_{x}(\bX_{\scT}(s),\bu_{\scT}(s))|^2\]ds \\
& \les C_1\int_{\t}^{\t+1}\[2C_0^2 + (1+2C_0^2)|B_1^{\top}\f_{\scT}(s)|^2\]ds \\
& =2C_1C_0^2 + C_1(1+2C_0^2)\Big\{\int_{[\t,\t+1]\setminus \Si_2}|B_1^{\top}\f_{\scT}(s)|^2ds + \int_{\Si_2}|B_1^{\top}\f_{\scT}(s)|^2ds \Big\}\\
&\les 2C_1C_0^2 + C_1(1+2C_0^2)N_2^2 + C_1(1+2C_0^2)\frac{M|B_1^{\top}|^2\p_{\scT}^2}{N_1^2}. 
\end{align*}
Now, choosing $N_1^2 = \frac{1}{2}C_1(1+2C_0^2)M|B_1^{\top}|^2$ implies
$$|\p_{\scT}|^2 \les 4C_1C_0^2 + 2C_1(1+2C_0^2)N_2^2,$$
which implies the uniform boundedness of $\f_{\scT}(\cd)$ on $[0,T-1]$.
Finally, by \rf{26-1-14:g} and \rf{26-1-13:eq-f}, there exists a constant $N>0$, independent of $T$, such that 
\begin{equation*}
|\bu_{\scT}(t)|\les N,\q\ae~t\in[0,T], \q\forall T>0.
\end{equation*}
%


{\it Step 3.} We now establish the main result \eqref{26-1-15:PETP-u}.
For $t\in[0,T]$, recall $\f_{\scT}(\cd)$ in Step 2 and introduce the notation 
\begin{align*}
&\tX(t)\deq\bX_{\scT}(t)-x^*, \q \wt\f_{\scT}(t)\deq\f_{\scT}(t)-P_1^{\top}\l_1^*,\q \tu(t) \deq \bu_{\scT}(t)-u^*,\\
&\xi_1(t)\deq f_{x}(\bX_{\scT}(t),\bu_{\scT}(t))-f_{x}(x^*,u^*), \q 
\xi_2(t)\deq f_{u}(\bX_{\scT}(t),\bu_{\scT}(t))-f_{u}(x^*,u^*).
\end{align*}
Combining \rf{Prob-O-xu}, \rf{26-1-14:xu_star}, \rf{26-1-14:stationary} and \rf{26-1-13:eq-f}, we have  
\begin{equation}\label{26-1-14:wtf}
\lt\{\begin{aligned}
&\dot{\wt\f}_{\scT}(t)= -A_{11}^{\top}\tf(t) - P_1^{\top}\xi_1(t), \q \ae~ t\in[0,T],\\
&B_1^{\top}\tf(t) + \xi_2(t) = 0,\q \ae~ t\in[0,T], 
\end{aligned}\rt. 
\end{equation}
where
$$|P_1^{\top}\xi_1(t)|\les |\xi_1(t)|,\q t\in[0,T].$$
Note that from \autoref{crllry:MS-Tp} and Step 2, $(\bX_{\scT}(\cd),\bu_{\scT}(\cd))$ is essentially uniformly bounded. 
By \ref{A2}, $\nabla f$ is continuously differentiable, so it is also locally Lipschitz.
Consequently, there exists constant $K>0$, independent of $T$, such that
\begin{equation}\label{26-1-14:xi12}
|\xi_1(t)|+|\xi_2(t)| \les K\[|\tX(t)|+|\tu(t)|\],\q \ae~t\in[0,T],\q \forall T>0.
\end{equation}
In what follows, we let $K>0$ denote a generic constant,
which is independent of $T$ and may vary from line to line. 
Recall the constants $\e, \a_0, T_0, S, L$ and the sequences $\{\a_i\}_{i=1}^{j}$, $\{\b_i\}_{i=1}^{j}$ from claim \rf{TpX-claim1}--\rf{Ineq-abl}.
Without loss of generality, we assume $T>2(\a_0+T_0+S)$, so that $j\ges 2$.
Then 
\begin{equation}\label{26-1-17:MS-Xu-i}
\int_{\a_i}^{\b_i}\[|\tX(s)|^2+|\tu(s)|^2\]ds \les L\e^{2i},\q i=1,2,\cdots,j.
\end{equation}
Consequently, for each $i=1,2,\cdots,j-1$, applying \autoref{lmm:X-bound} (ii) to \rf{26-1-14:wtf} and invoking \rf{26-1-14:xi12}--\rf{26-1-17:MS-Xu-i}, we obtain 
\begin{equation}\label{26-1-15:tfab}
\begin{aligned}
|\tf(\a_i)|^2+|\tf(\b_i)|^2 &\les K \int_{[\a_i,\a_i+\a_0]\cup[\b_i-\a_0,\b_i]}\[|B_1^{\top}\tf(s)|^2 + |P_1^{\top}\xi_1(s)|^2\]ds \\
&\les K\int_{\a_i}^{\b_i}\[|\tX(s)|^2+|\tu(s)|^2\]ds \les KL\e^{2i}.
\end{aligned}
\end{equation}

Next, since $(A_{11},B_1)$ is controllable, we can choose $F\in\dbR^{m\times k}$ such that $\hA_{11} \deq A_{11}+B_1F$ is stable.
Consequently, there exist a positive definite matrix $\vP\in\dbS^{n}_{+}$, such that 
$$\vP \hA_{11}^{\top} + \hA_{11}\vP + I_{k} =0.$$
Using the second equation in \rf{26-1-14:wtf}, the first equation can be written as
$$\dot{\wt\f}_{\scT}(t)=-\hA_{11}^{\top}\tf(t) - P_1^{\top}\xi_1(t) - F^{\top}\xi_2(t),\q \ae~t\in[0,T].$$
Differentiating $\lan \vP\tf(t),\tf(t)\ran$ yields
\begin{align*}
\frac{d}{dt}\lan \vP\tf(t),\tf(t)\ran & = -\lan (\vP\hA_{11}^{\top} + \hA_{11}\vP)\tf(t),\tf(t)\ran 
- 2\lan \vP\tf(t),P_1^{\top}\xi_1(t)+F^{\top}\xi_2(t) \ran \\
& \ges \frac{1}{2}|\tf(t)|^2 - K\[|\tX(t)|^2+|\tu(t)|^2\],\q \ae~t\in[0,T].
\end{align*}
For each $i=1,2,\cdots,j-1$, by integrating both sides of above inequality over $[\a_i,\b_i]$ and noting \rf{26-1-17:MS-Xu-i}--\rf{26-1-15:tfab}, we have
\begin{align*}
\int_{\a_i}^{\b_i}|\tf(s)|^2ds &\les 2K\int_{\a_i}^{\b_i}\[|\tX(s)|^2+|\tu(s)|^2\]ds \\
&\hp{=\ }+ 2\lan \vP\tf(\b_i),\tf(\b_i)\ran -2 \lan \vP\tf(\a_i),\tf(\a_i)\ran \les K\e^{2i}.
\end{align*}
Therefore, by \autoref{lmm:X-bound} (i) and combining \rf{26-1-14:xi12}--\rf{26-1-15:tfab}, we obtain 
\begin{align*}
\max_{t\in[\a_i,\b_i]}|\tf(t)|^2 & \les (2|\hA_{11}^{\top}|+1)\int_{\a_i}^{\b_i}\[|\tf(s)|^2 + |P_1^{\top}\xi_1(s)|^2\]ds + |\tf(\a_i)|^2\les K \e^{2i}.
\end{align*}
Since $\b_j-\a_{j}<2S$, we have
$$|\tf(t)|\les Ke^{i\ln\e},\q t\in[\a_i,\b_i], \q i=1,2,\cdots,j.$$
Applying the same procedure as in handling \rf{Sec4:Ineq_ab} and setting $\l \deq \frac{-\ln\e}{S}$, it follows that
$$|\tf(t)|\les K\[e^{-\l t} + e^{-\l (T-t)}\],\q t\in[0,T].$$
Finally, by \ref{A2}--\ref{A3}, \rf{26-1-14:xu_star} and \rf{Sec5:def-g}--\eqref{26-1-14:g}, Taylor's formula with integral remainder yields
\begin{align*}
|\bu_{\scT}(t)-u^*|& = |g(\bX_{\scT}(t),B_1^{\top}\f_{\scT}(t))-g(x^*,B_1^{\top}P_1^{\top}\l_1^*)| \\
& \les \g(\max\{\sqrt{M},|x^*|\})|\tX(t)| + \frac{|B_1^{\top}|}{\d}|\tf(t)|,\q \ae~t\in[0,T],
\end{align*}
which follows \rf{26-1-15:PETP-u} immediately.
\end{proof}

\begin{remark}
The proofs of \autoref{thm:PETP-X} and \autoref{thm:PETP-u} are inspired by the approach
developed in Lou--Wang~\cite{Lou-Wang2019}. In particular, two ingredients adapted from
\cite{Lou-Wang2019} play a crucial role in our analysis: (i) an induction argument based on a
carefully constructed family of subintervals, which leads to the claim that there exist
constants $S,L>0$, independent of $T$, and two finite sequences $\{\a_i\}_{i=1}^{j}$ and
$\{\b_i\}_{i=1}^{j}$ such that \rf{TpX-claim1}--\rf{Ineq-abl} hold; and (ii) uniform
a priori bounds for the optimal control, which allow us to convert integral estimates into
pointwise turnpike estimates. We would like to emphasize that the ideas introduced in
\cite{Lou-Wang2019} provide an elegant and powerful framework for studying turnpike
phenomena, and we gratefully acknowledge their influence on the present work.  
\end{remark}

\end{document}